\documentclass[11pt,a4paper,oneside]{article}

\usepackage[latin1]{inputenc}
\usepackage{amsmath,amsthm}
\usepackage[english]{babel}
\usepackage{latexsym}
\usepackage{amsfonts}
\usepackage{amssymb}
\usepackage{euscript}
\usepackage{color}
\usepackage{longtable}
\usepackage{tabularx}
\usepackage{pdfsync}
\usepackage{pstricks,pst-plot,pst-node,epsfig}
\usepackage{hyperref}
\usepackage{pgfplots}
\catcode`\@=11
\def\newremark#1{\@ifnextchar[{\@orem{#1}}{\@nrem{#1}}}
\def\@nrem#1#2{%
\@ifnextchar[{\@xnrem{#1}{#2}}{\@ynrem{#1}{#2}}}
\def\@xnrem#1#2[#3]{\expandafter\@ifdefinable\csname #1\endcsname
{\@definecounter{#1}\@addtoreset{#1}{#3}%
\expandafter\xdef\csname the#1\endcsname{\expandafter\noexpand
    \csname the#3\endcsname \@remcountersep \@remcounter{#1}}%
\global\@namedef{#1}{\@rem{#1}{#2}}\global\@namedef{end#1}{\@endremark}}}
\def\@ynrem#1#2{\expandafter\@ifdefinable\csname #1\endcsname
{\@definecounter{#1}%
\expandafter\xdef\csname the#1\endcsname{\@remcounter{#1}}%
\global\@namedef{#1}{\@rem{#1}{#2}}\global\@namedef{end#1}{\@endremark}}}
\def\@orem#1[#2]#3{\expandafter\@ifdefinable\csname #1\endcsname
    {\global\@namedef{the#1}{\@nameuse{the#2}}%
\global\@namedef{#1}{\@rem{#2}{#3}}%
\global\@namedef{end#1}{\@endremark}}}
\def\@rem#1#2{\refstepcounter
      {#1}\@ifnextchar[{\@yrem{#1}{#2}}{\@xrem{#1}{#2}}}
\def\@xrem#1#2{\@beginremark{#2}{\csname the#1\endcsname}\ignorespaces}
\def\@yrem#1#2[#3]{\@opargbeginremark{#2}{\csname
         the#1\endcsname}{#3}\ignorespaces}
\def\@remcounter#1{\noexpand\arabic{#1}}
\def\@remcountersep{.}
\def\@beginremark#1#2{\rm \trivlist \item[\hskip \labelsep{\bf #1\ #2}]}
\def\@opargbeginremark#1#2#3{\rm \trivlist
        \item[\hskip \labelsep{\bf #1\ #2\ (#3)}]}
\def\@endremark{\endtrivlist}
\catcode`\@=12
\newcommand{\biindice}[3]%
{

\begin{array}[t]{c}
#1\\
{\scriptstyle #2}\\
{\scriptstyle #3}
\end{array}

}
\newtheorem{Thm}{Theorem}[section]
\newtheorem{Cor}[Thm]{Corollary}
\newtheorem{Prop}[Thm]{Proposition}
\newtheorem{Lem}[Thm]{Lemma}
\newtheorem{Def}[Thm]{Definition}
\newremark{Rem}[Thm]{Remark}

\def\a{\alpha}
\def\b{\beta}

\def\l{\lambda}
\def\p{\partial}

\def\e{\varepsilon}
\def\v{\varphi}

\def\mc{\mathcal}
\def\mf{\mathfrak}
\def\ua{\uparrow}
\def\da{\downarrow}

\newcommand{\R}{\mathbb R}
\newcommand{\N}{\mathbb N}
\newcommand{\Z}{\mathbb Z}

\numberwithin{equation}{section}

\theoremstyle{definition}

\theoremstyle{plain}
\newtheorem{theorem}{Theorem}[section]

\theoremstyle{definition}

\title{\vskip-2.5cm
Multiplicity of subharmonics in a class of periodic predator-prey Volterra models
\thanks{This paper has been written under the auspices of the Ministry of
Science, Technology and Universities of Spain under Research Grant MTM2015-65899-P, and of the IMI of
Complutense University.
}
}

\author{
\sc
Julián L\'opez-G\' omez
\\
\small
Universidad Complutense de Madrid
\\
\small
Instituto de Matem\'{a}tica Interdisciplinar (IMI)
\\
\small
Departamento de An\'alisis Matem\'atico y Matem\'atica Aplicada
\\
\small  Plaza de las Ciencias 3, 28040   Madrid, Spain
\\
\small E-mail: {\tt   julian@mat.ucm.es}
\medskip
\\
\sc
Eduardo Muñoz-Hernández
\\
\small
Universidad Complutense de Madrid
\\
\small
Departamento de An\'alisis Matem\'atico y Matem\'atica Aplicada
\\
\small  Plaza de las Ciencias 3, 28040   Madrid, Spain
\\
\small E-mail: {\tt eduardmu@ucm.es }
}

\date{\today}

\numberwithin{equation}{section}

\begin{document}

\maketitle

\begin{abstract}
\noindent
  This paper ascertains the global topological structure of the set of subharmonics of arbitrary order of the
  periodic predator-prey model introduced in \cite{LOT96}. By constructing
   the iterates of the monodromy operator of the system, it is shown that the system admits subharmonics of all orders for the
  appropriate ranges of values of the parameters. Then, some sharp results of topological nature in the context of global bifurcation theory provide us with the fine topological structure of the components of subharmonics emanating
  from the $T$-periodic coexistence state.
 Structure of the set of bifurcation points.
Global bifurcation diagrams. 
\end{abstract}


\section{Introduction}
In this paper we analyze the global structure of the set of subharmonics of the periodic predator-prey model
\begin{equation}
\label{1.1}
\left\{ \begin{array}{l} u'=\a(t)u(1-v) \\ v'=\b(t)v(-1+u) \end{array} \right.
\end{equation}
where $\a(t)$ and $\b(t)$ are real continuous $T$-periodic functions such that
\begin{equation}
\label{1.2}
  \a= 0\quad \hbox{on}\;\; [\tfrac{T}{2},T], \quad \b = 0 \quad \hbox{on}\;\; [0,\tfrac{T}{2}],
\end{equation}
$\a(t)>0$ if $t \in (0,\tfrac{T}{2})$, and $\b(t)>0$ if $t\in (\tfrac{T}{2},T)$, which entail $\a\b=0$.  This model was introduced by J. L\'opez-G\'omez, R. Ortega and A. Tineo \cite{LOT96} as a simple  example of a predator-prey model with an unstable coexistence state. Later, it was shown in \cite{LG00} that it actually admits three $T$-periodic (non-degenerate) coexistence states: one $T$-periodic and two additional $2T$-periodic solutions. The non-degeneration of these solutions facilitated the construction of some examples of $T$-periodic Lotka--Volterra models
\begin{equation}
\label{1.3}
  \left\{ \begin{array}{l} u'=\l(t) u - a(t) u^2 - b(t) uv  \\ v'=-\mu(t)v +c(t)uv-d(t)v^2 \end{array} \right.
\end{equation}
with at least three coexistence states (see \cite{LG00}). In \eqref{1.3},
$\l, \mu, a, b, c, d$ are smooth positive $T$-periodic functions.
Such multiplicity results contrast very strongly with the main theorem of J. L\'opez-G\'omez and R.  Pardo \cite{LP93}, where it
was established the uniqueness of the coexistence state for the boundary value problem
\begin{equation}
\label{1.4}
  \left\{ \begin{array}{ll} \begin{array}{l} -u''= \l(x) u-a(x)u^2-b(x)uv\\ -v''=-\mu(x) v+c(x)uv-d(x)v^2 \end{array} & \quad \hbox{in}\;\; (0,L), \\
  u(0)=u(L)=v(0)=v(L)=0, & \end{array} \right.
\end{equation}
where $\l, \mu, a, b, c, d$ are  positive (arbitrary) continuous functions in $[0,L]$. Inheriting the same non-cooperative structure, at first glance causes some perplexity that \eqref{1.3} and \eqref{1.4} behave so differently.
The original theorem of \cite{LP93} was later refined in a series of papers by A. Casal et al. \cite{CEL}, E. N. Dancer et al. \cite{DLO} and J. L\'opez-G\'omez and R. Pardo \cite{LP98}.
\par
An important feature of model \eqref{1.1} is that it does not fit within the general setting of  T. Ding and F. Zanolin \cite{DZ96},
where the existence of higher order subharmonics for a general class of predator-prey models was established. Precisely, \cite[Th.3]{DZ96} gives some general conditions on the nonlinearities $f(t,v)$ and $g(t,u)$ so that the Lotka--Volterra predator-prey system
\begin{equation}
\label{1.5}
\left \{
\begin{array}{ll}
u'=uf(t,v)\\
v'=vg(t,u)
\end{array}
\right.
\end{equation}
can admit higher order subharmonics. In \eqref{1.5},  $f(t,v)$ and $g(t,u)$ are continuous functions $T$-periodic in time, $t$, satisfying certain bounds for the existence of T-periodic solutions and such that, for every $t\in [0,T]$, either $v\mapsto f(t,v)$ is (strictly) decreasing, or
$u\mapsto g(t,u)$ is (strictly) increasing. Under these assumptions,  \cite[Th. 3]{DZ96} establishes the existence of an integer $m^*\geq 2$ such that \eqref{1.5} admits, at least, one $mT$-periodic solution for all $m\geq m^*$.
\par
Although setting
\[
  f(t,v):=\a(t)(1-v),\qquad g(t,u):=\b(t)(-1+u),
\]
\eqref{1.1} can be also written down in the form of \eqref{1.5}, by \eqref{1.2}, neither $\a(t)(1-v)$ can be decreasing for all $t\in [0,T]$,  nor $\b(t)(-1+u)$ can be increasing for all $t\in[0,T]$. Thus, \eqref{1.1} remains outside the class of models considered in \cite{DZ96}. In particular,  \cite[Th. 3]{DZ96} cannot be applied to establish the existence of higher order subharmonics for \eqref{1.1}.
\par
The main goal of this paper is to construct the set of all subharmonics of \eqref{1.1} in the special,
but extremely interesting case, when
\begin{equation}
\label{1.6}
  A := \int_0^T \a(t)\,dt = \int_0^T \b(t)\,dt>0.
\end{equation}
Precisely,  it will be shown that, under assumption \eqref{1.6},  the model \eqref{1.1} admits subharmonics of any order for the appropriate range of values of $A>0$, which will be regarded as a bifurcation parameter throughout this paper. Actually, our analysis establishes  the existence of an integer $m^*(A)\geq 1$ such that \eqref{1.1} possesses, at least, \emph{two} subharmonic solutions of order $m$ for all $m\geq m^*(A)$. In particular, \cite[Th. 3]{DZ96} seems to be true in much more general situations than those originally dealt with in \cite{DZ96}. Moreover, as a direct consequence of our analysis,
\begin{equation}
\label{1.7}
  \lim_{A\da 0}m^*(A)=+\infty,\quad \hbox{whereas}\;\; m^*(A)=2 \;\; \hbox{if}\;\; A>2.
\end{equation}
Figure \ref{Fig9} summarizes, at a glance, the main findings of this paper. It is an sketch of the global bifurcation diagram of subharmonics, where we are plotting the value of $A$ in abscisas versus the value of $x=u_0=v_0$ in ordinates. Naturally,
\[
  (u_0,v_0)=(u(0),v(0))
\]
stands for the initial condition of \eqref{1.1}.
\begin{figure}[h!]
		\centering
		\includegraphics[scale=0.72]{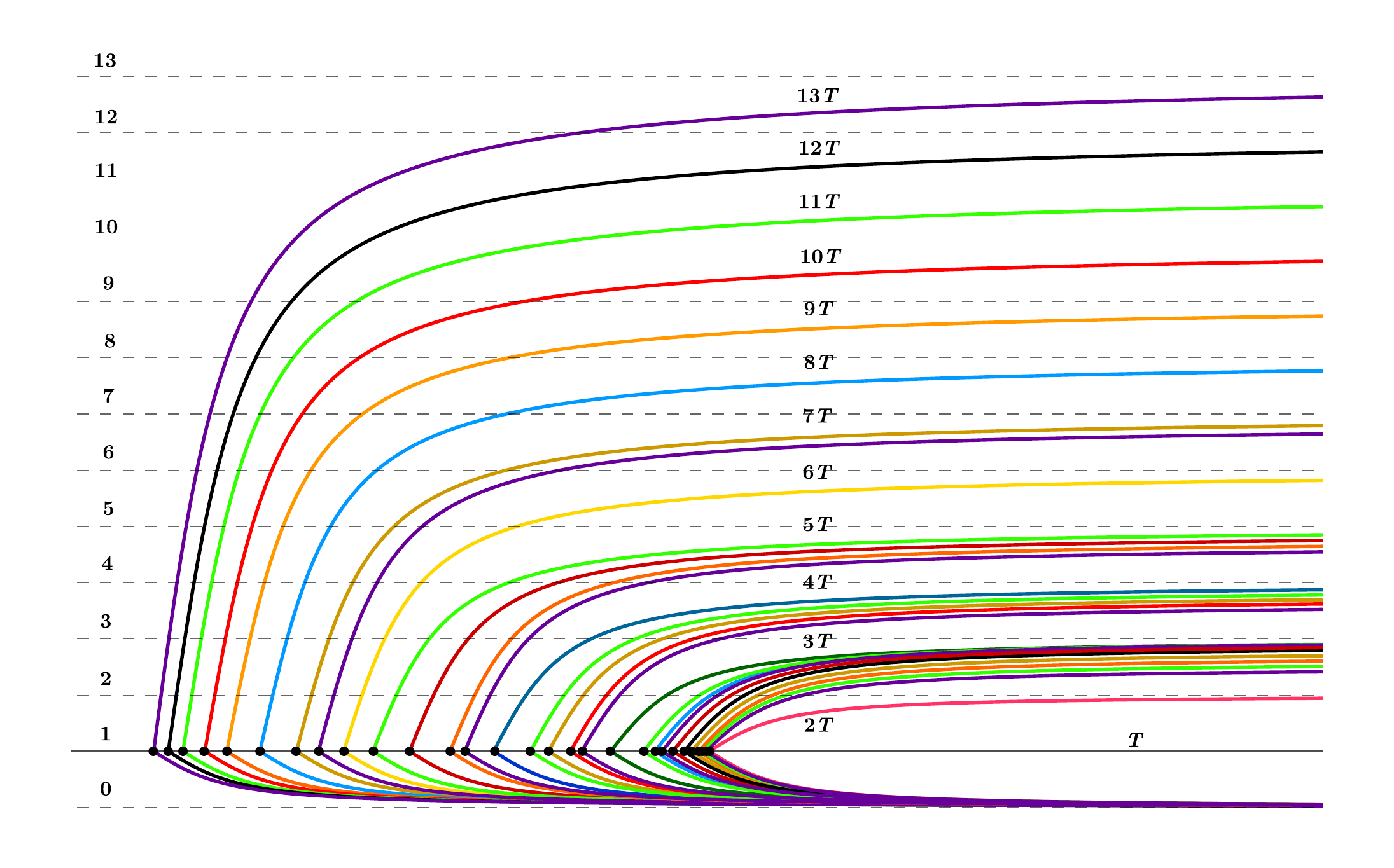}
		\caption{Subharmonics of \eqref{1.1} under condition \eqref{1.6} with $x=u_0=v_0$. }
		\label{Fig9}
\end{figure}
Each of the curves plotted in Figure \ref{Fig9} represents a component of $nT$-periodic coexistence states of \eqref{1.1} for each integer $n\geq 1$. By a component it is meant a closed and connected subset of the solution set of \eqref{1.1} which is maximal for the inclusion. Each point on the corresponding line, $(A,x)$, provides us with a value of $A$
for which \eqref{1.1} admits a $nT$-periodic solution with $u_0=v_0=x$. By the intrinsic nature of \eqref{1.1}, it turns out that all these components are separated from each other. By some existing results of topological nature  in global bifurcation theory, all of them have an unbounded $A$-projection. However, except for the first three, whose local bifurcation diagrams are described by  Theorem \ref{th6.1}, the nature of their local bifurcations from $(A,1)$ is not known yet, being possibly random. The problem of ascertaining weather, or not, this occurs, seems extremely challenging.  Note that
\[
  (A,u,v)=(A,1,1)
\]
solves \eqref{1.1} for all $A>0$. More precisely, Figure \eqref{Fig9} shows all the components of subharmonics of order $n$ of \eqref{1.1} emanating from the straight line $(A,1)$ for $1\leq n\leq 13$. It contains the plots of:
\begin{itemize}
\item  $1$ component of subharmonics of order $1$;
\item $1$ component of subharmonics of order $2$;
\item $1$ component of subharmonics of order $3$;
\item $2$  components of subharmonics of order $4$; one of them
is actually the component of subharmonics of order $2$;
\item $2$  components of subharmonics of order $5$;
\item $3$ components of subharmonics of order $6$; one of them is actually  the component of
subharmonics of order $2$ and another  must be the component of order $3$;
\item $3$ components of subharmonics of order $7$;
\item $4$ components of subharmonics of order $8$; one of them must be the component
of order $2$ and another one is a component of subharmonics with minimal order $4$;
\end{itemize}
and so on... The fact that the number of components of subharmonics of order $n\geq 1$ grows to
$+\infty$ as $n\ua +\infty$ is rather intriguing and it seems inherent to the non-cooperative character of \eqref{1.1} and attributable to the $T$-periodicity of $\a(t)$ and $\b(t)$. The emergence of secondary bifurcations in any of these components cannot be a priori excluded, however no higher order bifurcations have been represented in Figure \ref{Fig9}.
\par
Thanks to Theorems \ref{th4.8} and Theorem \ref{th6.1}, for every $n\geq 1$, the bifurcation points from $(A,1)$ to the $nT$-periodic coexistence states of \eqref{1.1} are given by the positive roots of the polynomial
\begin{equation}
\label{1.8}
	p_n(A):=[2-(-1)^n A]p_{n-1}(A)-p_{n-2}(A),\qquad n\geq 3,
\end{equation}
where
\[
  p_1(A):= 1,\qquad p_2(A):=2-A.
\]
Although, according to  Theorem \ref{th5.2}, for every $n\geq 2$, the positive roots of $p_{2n}(A)$ are separated by the positive roots of $p_{2n-1}(A)$,  the positive roots of $p_{2n+1}(A)$ are separated by those of $p_{2n}(A)$ less than $2$, and, for every $n\geq 1$, the (even) polynomials  $\tfrac{p_{2n}(A)}{2-A}$ and $p_{2n-1}(A)$ have (exactly) $n-1$ positive roots,  which are real and algebraically simple,  the problem of ascertaining the sharp ordering structure, if any,  of the set of all these
positive roots, which is a numerable subset of $(0,2]$, remains an open problem in this paper. Although there are some serious evidences that this set should be dense in the interval $[0,2]$, a rigorous proof of this feature is not available yet.
\par
The fact that  the positive roots of  $p_n(A)$ are algebraically simple allows us to apply the main  theorem of M. G. Crandall and P. H. Rabinowitz \cite{CR71} to prove that each of the components of subharmonics in Figure \ref{Fig9} must be a real analytic curve about  their bifurcation points from $(A,1)$.
\par
The mathematical analysis carried out in this paper has been tremendously facilitated  by the fact that $\a\b=0$, which provides us with a rather explicit formula for the iterates, $\mc{P}_n$, $n\geq 2$, of the
monodromy operator, $\mc{P}_1$. Thanks to Proposition \ref{pr3.2},
for every $ n\geq 2 $, the Poincar\'e map $\mc{P}_n$ can be expressed through
\begin{equation}
\label{1.9}
	(u_n,v_n)=\mc{P}_n(x,x)=\left(xE_{2n-1}(x),xE_{2n}(x)\right),
\end{equation}
where
\begin{equation}
\label{1.10}
\left\{
\begin{array}{lll}
E_0(x):=1,\quad E_1(x):=e^{(1-x)A},  \\[4pt]
E_n(x):=\left\{ \begin{array}{ll} e^{[x(E_1(x)+E_3(x)+\cdots+E_{n-1}(x))-\frac{n}{2}]A} & \quad \mathrm{if\ } \;\; n\in 2\N, \\[4pt]
e^{[\frac{n+1}{2}-x(E_0(x)+E_2(x)+\cdots+E_{n-1}(x))]A} & \quad \mathrm{if\ } \;\; n\in 2\N+1. \end{array} \right.
\end{array}
\right.
\end{equation}
Thanks to Theorem \ref{th3.3}, for every $n\geq 2$, the positive fixed points of $\mc{P}_n$, which provide us with the $nT$-periodic coexistence states of \eqref{1.1}, are given by the zeros of the map
\begin{equation}
\label{1.11}
	\varphi_n(x)= \varphi_{n-1}(x) - 1 + xE_{n-1}(x), \qquad x\in [0,n].
\end{equation}	
Thus,
\begin{align*}
  \v_1(x) & =x-1, \\ \v_2(x) & = \v_1(x)-1+x e^{(1-x)A}, \\
  \v_3(x)  &=  \v_2(x)-1 +xe^{(xe^{(1-x)A}-1)A},\\
  \varphi_4(x) & =\v_3(x)-1+ x e^{(2-x-xe^{(xe^{(1-x)A}-1)A})A},\\
  \varphi_5(x) & = \v_4(x)-1+ x e^{\left(x e^{(1-x)A}+xe^{(2-x-xe^{(xe^{(1-x)A}-1)A})A}-2\right)A}.
\end{align*}
In particular, the positive fixed points of $\mc{P}_5$ are given by the positive zeros of $\v_5(x)$, which consists, essentially, in the composition of $4$ exponentials functions. This circumstance might provoke dramatic oscillations of
$\v_5(x)$ between some consecutive positive zeros. For instance, choosing $A=5$ and $x=0.1$, it turns out that $\v_5(0.1)\sim 10^{30}$, which lies outside  the precision range of most of personal computers. Therefore, without no further work, numerics cannot be of any help in constructing the global bifurcation diagram sketched in Figure \ref{Fig9}.
\par
Lastly, we will consider the associated perturbed $T$-periodic functions
\begin{equation}
\label{1.12}
   \a_\e:=\a+\e,\qquad \b_\e = \b +\e,
\end{equation}
where $\e>0$, as well as the associated predator-prey model
\begin{equation}
\label{1.13}
\left\{ \begin{array}{l} u'=\a_\e(t)u(1-v), \\ v'=\b_\e(t)v(-1+u). \end{array} \right.
\end{equation}
Taking $\e=0$ in \eqref{1.13} gives \eqref{1.1}. Although the $nT$-periodic coexistence states
of \eqref{1.1} might degenerate, thanks to a celebrated result of A. Sard \cite{Sard}, most of the subharmonics of order $n$
of \eqref{1.1} should provide us with subarmonics of order $n$ of \eqref{1.13} for sufficiently small $\e>0$. Therefore, the global topological structure sketched by Figure \ref{Fig9} should be essentially preserved, at least for sufficiently small $\e>0$. Note that \cite[Th. 3]{DZ96} applies to \eqref{1.13} for all $\e>0$,
because $\a_\e(t)>0$ and $\b_\e(t)>0$ for all $t\in [0,T]$. Thus, it is rather natural to conjecture that, actually, Figure
\ref{Fig9} provides us with the minimal admissible  complexity of the set of subharmonics of \eqref{1.13}
for sufficiently small $\e>0$. Like in \cite{LG00}, these multiplicity results should provide us with a series of (very intriguing) multiplicity results for \eqref{1.4}.
\par
The distribution of this paper is the following. Section 2 studies the structure and multiplicity of the low order subharmonics of \eqref{1.1} in the general case when
\[
  0 < A := \int_0^T\a(t)\,dt \neq B := \int_0^T \b(t)\,dt>0.
\]
It substantially sharpens some previous findings of \cite{LG00} by establishing the exact multiplicity of
the $2T$-periodic solutions of \eqref{1.1} when $AB>4$. The rest of the paper focuses attention into the special, but extremely important case, when $A=B$. In Section 3 we construct the Poincar\'e maps $\mc{P}_n$ for all $n\geq 1$. In Section 4 we introduce the associated polynomials
\begin{equation*}
  p_n(A):=  \frac{d\v_n (A,1)}{dx}=\mf{L}(n;A),\qquad A>0,
\end{equation*}
whose positive roots provide us with the bifurcation points to subharmonics from $(A,1)$, and analyze
some of their most fundamental properties. In Section 5 we establish some fundamental separation properties between
the zeros of these polynomials and show that all their positive roots are algebraically simple. This property has important
consequences from the point of view of local and global bifurcation theory. Finally, in Section 6 we derive and discuss the global bifurcation diagram sketched in Figure \ref{Fig9}.

\section{Multiplicity and structure of $T$-periodic and $2T$-periodic solutions in the model of \cite{LG00}}

According to  \cite[Th. 5.1]{LG00}, $(u,v)=(1,1)$ provides us with the unique $T$-periodic solution of \eqref{1.1}, and \eqref{1.1} admits, at least, two $2T$-periodic coexistence states if, and only if, $AB>4$, where
\begin{equation}
\label{2.1}
  A := \int_0^T \a(s)\,ds>0,\qquad B:=\int_0^T \b(s)\,ds>0.
\end{equation}
The next result sharpens these findings.

\begin{theorem}
\label{th2.1}
Suppose $AB>4$. Then, the problem \eqref{1.1} possesses exactly two $2T$-periodic coexistence states (with minimal period $2T$, of course).
\end{theorem}

\begin{proof}
We proceed as in the proof  \cite[Th. 5.1]{LG00}. Since $\a\b=0$ in $\R$, the system
\eqref{1.1} can be solved. Actually, for every $(u_0,v_0)\in\R^2$, the unique solution of \eqref{1.1}, $(u,v)$, such that
$(u(0),v(0))=(u_0,v_0)$ is given by
\begin{equation}
\label{2.2}
  u(t)= u_0 e^{(1-v_0)\int_0^t\a},\qquad v(t)=v_0 e^{(u(T)-1)\int_0^t\b}, \qquad t\in \R.
\end{equation}
Thus, the associated $T$-time and $2T$-time Poincar\'e maps, $\mathcal{P}_1$ and $\mathcal{P}_2$, are given by
\begin{equation}
\label{2.3}
  (u_1,v_1)  :=\mathcal{P}_1(u_0,v_0),\qquad u_1:=u_0e^{(1-v_0)A},\qquad v_1:= v_0e^{(u_1-1)B},
\end{equation}
and
\begin{equation}
\label{2.4}
  (u_2,v_2)  := \mathcal{P}_2(u_0,v_0)=\mathcal{P}^2_1(u_0,v_0)= \mathcal{P}_1(u_1,v_1)=
  \left( u_1e^{(1-v_1)A}, v_1e^{(u_2-1)B}\right).
\end{equation}
Thus, substituting \eqref{2.3} into \eqref{2.4} yields
\begin{equation}
\label{2.5}
  u_2= u_0 e^{(2-v_0-v_1)A},\qquad v_2= v_0 e^{(u_1+u_2-2)B}.
\end{equation}
A solution  with initial data $(u_0,v_0)$ provides us with a componentwise positive fixed point of $\mathcal{P}_2$ if, and only if,
$u_0>0$, $v_0>0$, $v_0+v_1=2$ and $u_1+u_2=2$. Hence, since $u_2=u_0$, this is equivalent to
\begin{equation}
\label{2.6}
   u_0>0,\qquad v_0>0,\qquad v_0+v_1=2,\qquad u_0+u_1=2.
\end{equation}
Note that, owing to \eqref{2.6},
\[
  0<u_0, u_1 <2,\qquad 0<v_0,v_1<2.
\]
Since $u_1=2-u_0$ and $v_1=2-v_0$, from \eqref{2.3} it becomes apparent that
\[
   2-u_0=u_0e^{(1-v_0)A},\qquad 2-v_0= v_0e^{(1-u_0)B}.
\]
Consequently,
\begin{equation}
\label{2.7}
 u_0=\frac{2}{e^{(1-v_0)A}+1},\qquad 2-v_0= v_0e^{\left(1-\frac{2}{e^{(1-v_0)A}+1}\right)B}=
 v_0e^{\frac{e^{(1-v_0)A}-1}{e^{(1-v_0)A}+1}B}
\end{equation}
and therefore, the $2T$-periodic coexistence states are given by the interior zeros of the map
\begin{equation}
\label{2.8}
  \v(x):= x\left( e^{\frac{e^{(1-x)A}-1}{e^{(1-x)A}+1}B}+1\right)-2,\qquad x\in [0,2].
\end{equation}
As this function satisfies $\v(0)=-2<0$, $\v(1)=0$, $\v(2)>0$ and
\[
  \v'(1)= 2-\frac{AB}{2}<0,
\]
because we are assuming that $AB>4$, it is easily seen that $\v(x)$ possesses, at least, besides $1$, two zeros, $z_1\in (0,1)$ and
$z_2\in (1,2)$. Note that $1$ provides us with the (unique) $T$-periodic solution of \eqref{1.1}. That these zeros are unique is based on the fact that any critical point of $\v$ on $(0,1)$, $x$, must satisfy $\v''(x)<0$, and hence, it is a quadratic local maximum, while $\v''(y)>0$ for all critical point, $y$, of $\v$ in $(1,2)$. In particular, since $\v(0)<0$ and $\v'(1)<0$, this entails
that $z_1$ is simple and, actually, $\v'(z_1)>0$, for as, otherwise, $\v(x)$ should have a local minimum
in $(0,1)$, which is impossible. Similarly, $\v'(z_2)>0$. In order to show the previous claim, suppose
\[
  \v'(x)=0 \quad \hbox{for some\;\;} x\in (0,2).
\]
Then,
\begin{equation}
\label{2.9}
  \v'(x)=  e^{\tfrac{e^{(1-x)A}-1}{e^{(1-x)A}+1}B}\left( 1- \frac{2ABxe^{(1-x)A}}{[e^{(1-x)A}+1]^2}\right)+1=0.
\end{equation}
Moreover, differentiating $\v'$ and rearranging terms yields
\begin{equation}
\label{2.10}
\v''(x)=e^{\tfrac{e^{(1-x)A}-1}{e^{(1-x)A}+1}B}\left[x\left(\tfrac{2ABe^{(1-x)A}}{[e^{(1-x)A}+1]^2}\right)^2 - \tfrac{4ABe^{(1-x)A}}{[e^{(1-x)A}+1]^2} + 2A^2Bxe^{(1-x)A}\tfrac{1-e^{(1-x)2A}}{[e^{(1-x)A}+1]^4}\right].
\end{equation}
Now, after some straightforward manipulations, it is easily seen that  \eqref{2.9} implies
\begin{equation}
\label{2.11}
\tfrac{2 \Big(e^{-\tfrac{e^{(1-x)A}-1}{e^{(1-x)A}+1}B}+1\Big)}{x}=\tfrac{4ABe^{(1-x)A}}{[e^{(1-x)A}+1]^2}, \qquad
\Big(\tfrac{e^{-\tfrac{e^{(1-x)A}-1}{e^{(1-x)A}+1}B}+1}{x}\Big)^2=\left(\tfrac{2ABe^{(1-x)A}}{[e^{(1-x)A}+1]^2}\right)^2,
\end{equation}
and substituting \eqref{2.11} into \eqref{2.10} we find that
\begin{align*}
\varphi''(x)&=e^{\tfrac{e^{(1-x)A}-1}{e^{(1-x)A}+1}B}\Big[ x\Big(\tfrac{e^{-\tfrac{e^{(1-x)A}-1}{e^{(1-x)A}+1}B}+1}{x}\Big)^2 -\tfrac{2\Big(e^{-\tfrac{e^{(1-x)A}-1}{e^{(1-x)A}+1}B}+1\Big)}{x}\Big] \\& \hspace{6cm} + e^{\tfrac{e^{(1-x)A}-1}{e^{(1-x)A}+1}B}\left[ 2A^2Bxe^{(1-x)A} \tfrac{1-e^{(1-x)2A}}{[e^{(1-x)A}+1]^4}\right] \\&=e^{\tfrac{e^{(1-x)A}-1}{e^{(1-x)A}+1}B}\Big[
\tfrac{e^{-2\tfrac{e^{(1-x)A}-1}{e^{(1-x)A}+1}B} - 1}{x} + 2A^2Bxe^{(1-x)A}\tfrac{1-e^{(1-x)2A}}{[e^{(1-x)A}+1]^4}\Big].
\end{align*}
Suppose $x\in(0,1) $. Then, the following holds
\[
  e^{-2\tfrac{e^{(1-x)A}-1}{e^{(1-x)A}+1}B}-1<0,\qquad  1-e^{(1-x)2A}<0.
\]
Therefore, $ \varphi''(x)<0 $, as claimed above.
\par
Suppose $x\in (1,2]$. Then,
\[
  e^{-2\tfrac{e^{(1-x)A}-1}{e^{(1-x)A}+1}B}-1>0,\qquad  1-e^{(1-x)2A}>0,
\]
and hence, $\v''(x)>0$, as requested. The proof is completed.
\end{proof}

According to the proof of Theorem \ref{th2.1}, if $AB>4$ then $\v(x)$ has exactly three (simple) zeros in $(0,2)$, $z_1, z_2, z_3$, such that $z_1\in (0,1)$, $z_2\in (1,2)$ and $z_3=1$, whereas
$$
  \v'(1) = 2-\frac{AB}{2} \geq 0 \quad \hbox{if}\;\; AB\leq 4,
$$
and hence, $1$ is the unique zero of $\v$ in this case.  Note that if $AB=4$, then $\v'(1)=0$ and
$$
  \varphi''(1)=\left(\frac{AB}{2}\right)^2-AB=4-4=0.
$$
Moreover, differentiating twice yields
\[
\v'''(x)=e^{q}\left(3(q')^2+3q''+x\left[(q')^3+3q'q''+q'''\right]\right),\quad
q(x):=e^{\tfrac{e^{(1-x)A}-1}{e^{(1-x)A}+1}B}, \quad x\in [0,2].
\]
Thus,
$$
  \varphi'''(1)=\tfrac{AB(5AB+2A^2)}{8}>0
$$
and therefore, $1$ is a treble zero of $\v(x)$ if $AB=4$. On the other hand,  the function $\v(x)$ can be also regarded as an analytic function of $x$ that varies continuously with $B>0$ and does not vanish at the ends of $[0,2]$. By Rouch\'e's theorem, $\v$ must have three zeros, counting orders, for every $B>0$. As $1$ is the unique real zero of $\v(x)$ if $AB< 4$ and $\v'(1)>0$ in this range,
it becomes apparent that $\v(x)$ possesses two complex zeros if $AB<4$. Those complex solutions are not going to be taken into account
throughout this paper.
\par
Subsequently, we are going to regard $B$ as the main continuation parameter in problem \eqref{1.1}. According to our previous analysis,
we already know that $(1,1)$ is the unique $2T$-periodic solution of \eqref{1.1} if $B<4/A$ (note that the minimal period of this solution is $T$), whereas \eqref{1.1} possesses (exactly) three $2T$-periodic solutions for every $B>4/A$. Moreover, two of them, those with minimal period $2T$, bifurcate from $(1,1)$ as the parameter $B$ crosses the critical value $4/A$, as it will become apparent later.  Precisely, we regard the solutions of \eqref{1.1} as solutions of
\begin{equation}
\label{2.12}
  0=\v(B,x):=  x\left( e^{\frac{e^{(1-x)A}-1}{e^{(1-x)A}+1}B}+1\right)-2,\qquad x\in [0,2],
\end{equation}
for some $B>0$. Note that $(B,x)=(B,1)$ is a solution curve of \eqref{2.12} defined for all $B>0$. Moreover, the linearization of
\eqref{2.12} at $(B,1)$ is
\[
  \mathfrak{L}(B)= \frac{d\v(B,1)}{dx} =2-\frac{AB}{2}
\]
which establishes an isomorphism of $\R$, unless $B=4/A$. Thus, this is the unique value of the parameter where bifurcation to
$2T$-periodic solutions of \eqref{1.1} can occur from $(1,1)$. Since
\[
  N[\mathfrak{L}(4/A)]= \R = \mathrm{span\,}[1]
\]
and
\[
  \mathfrak{L}_1  := \frac{d \mathfrak{L}(4/A)}{d B} = -\frac{A}{2}\neq 0,
\]
it becomes apparent that
\begin{equation}
\label{2.13}
  \mathfrak{L}_1 1 \notin R[\mathfrak{L}(4/A)]=[0].
\end{equation}
Hence, by the main theorem of M. G. Crandall and P. H. Rabinowitz \cite{CR71}, there exist $s_0>0$ and two analytic maps,
$x, B : (-s_0,s_0)\to \R$ such that $x(0)=1$, $B(0)=4/A$, $x(s)=1+s+\mc{O}(s^2)$ as $s\to 0$, and $\v(B(s),x(s))=0$ for every $s\in (-s_0,s_0)$. Moreover, except for $x=1$, these are the unique solutions of $\v(B,x)=0$ in a neighborhood of $(B,x)=(4/A,1)$. As due to Theorem \ref{th2.1}, $\v(B,x)=0$ cannot admit a solution $x\neq 1$ if $B\leq 4/A$, it becomes apparent that $B(s)>4/A$ for all $s\in (-s_0,s_0)$. Note that $x(s)>1$ if $s\in (0,s_0)$, while $x(s)<1$ if $s\in (-s_0,0)$. On the other hand, it readily follows from
\eqref{2.12} that $\v(B,x)<0$ if $x\leq 0$ and $\v(B,x)>0$ if $x\geq 2$. Thus, any solution of $\v(B,x)=0$ satisfies $x \in (0,2)$. In particular, $x(s) \in (0,2)$ for all $s\in(-s_0,s_0)$. Thus, as owing to \cite[Th. 5.2]{LG00} any solution, $(B,x)$, of $\v=0$ with $B>4/A$ is non-degenerated, by a rather standard continuation argument involving the Implicit Function Theorem the next result holds true.

\begin{theorem}
\label{th2.2}
  The set of zeros $(B,x)$ of $\v=0$ with $x\neq 1$, consists of a (global) analytic curve, $(B(s),x(s))$, $s\in\R$, such that $x(s)\in (0,2)$ for all $s\in\R$ and $B(\R)=(4/B,+\infty)$, much like illustrated by Figure \ref{Fig1}. Actually, each of the two half-branches, the upper and the lower ones,  can be globally parameterized by $B\in (4/A,+\infty)$.
\end{theorem}

\begin{figure}[h!]
	\centering
	\includegraphics[scale=0.37]{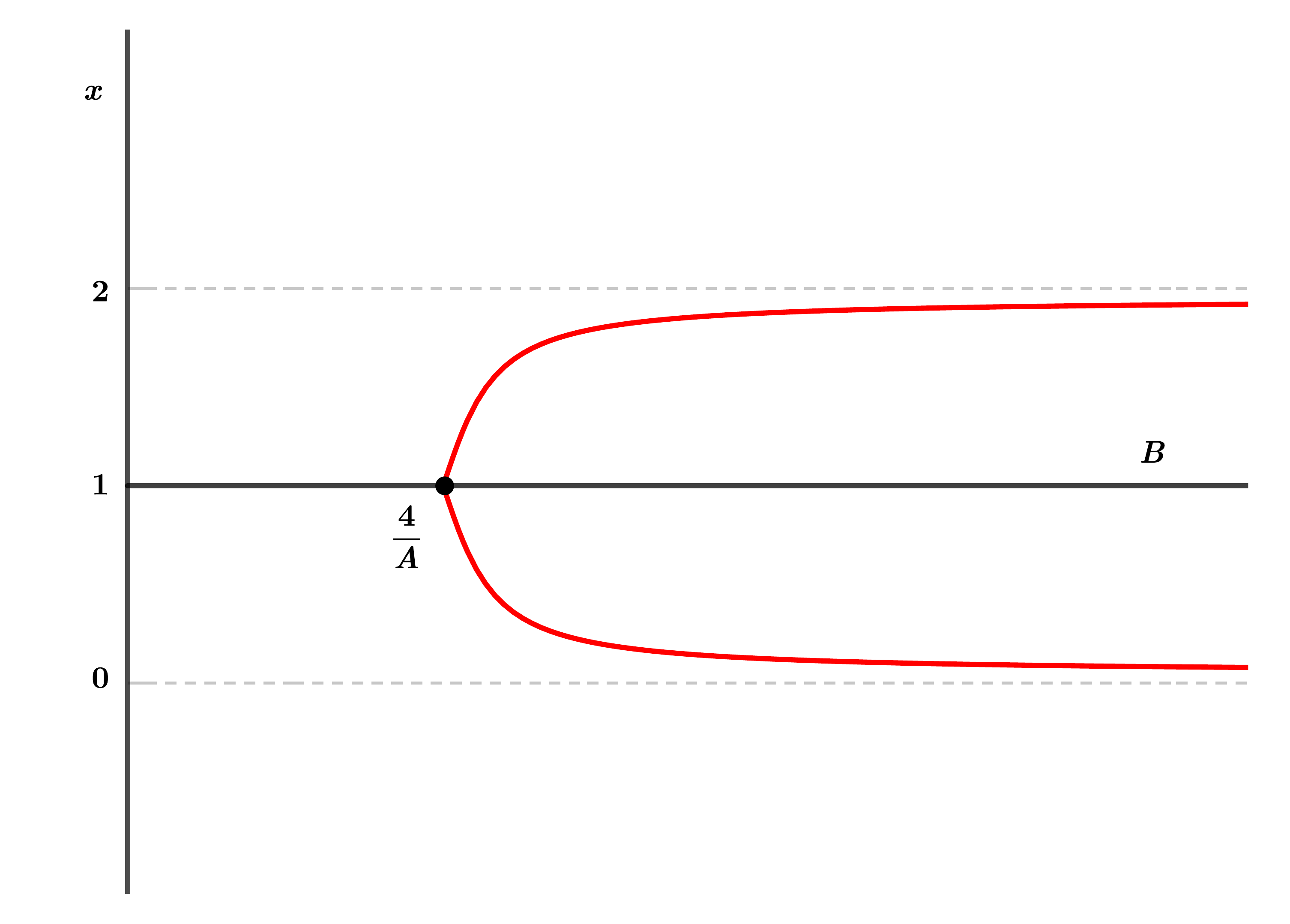}
	\caption{The set of $2T$-periodic solutions of \eqref{1.1}.}
	\label{Fig1}
\end{figure}

Since $\v(B,x)=0$ can be equivalently written down as
\[
  \dfrac{2}{x} -1=e^{\tfrac{e^{(1-x)A}-1}{e^{(1-x)A}+1}B},
\]
letting $B\to +\infty$ in this identity, it becomes apparent that
\[
  \lim_{B\ua +\infty} x = \left\{ \begin{array}{ll} 0 & \quad \hbox{if}\;\; x\in (0,1), \\
  2  & \quad \hbox{if}\;\; x\in (1,2),\end{array}\right.
\]
which is reflected in the global bifurcation diagram of Figure \ref{Fig1}.

\section{Constructing the $nT$-Poincar\'e maps}

Throughout the rest of this paper, for every integer $n\geq 1$, we denote by $\mc{P}_n$ the $nT$-Poincar\'e map of  \eqref{1.1}, and, for every
initial data $(u_0,v_0)$, with $u_0>0$ and $v_0>0$, we set
\begin{equation}
\label{3.1}
  (u_n,v_n):=\mathcal{P}_n (u_0,v_0) = \mathcal{P}_1^n (u_0,v_0)=(u_0,v_0).
\end{equation}
Then, iterating \eqref{2.3} $n$ times, it becomes apparent that
\begin{equation}
\label{3.2}
  (u_n,v_n)=\left( u_0 e^{(n-v_0-v_1-\cdots -v_{n-1})A},v_0e^{(u_1+u_2+\cdots + u_n -n)B}\right), \qquad n\geq 1.
\end{equation}
Consequently, the solution of \eqref{1.1} with initial data $(u_0,v_0)$, with $u_0>0$ and $v_0>0$,
provides us with a $nT$-periodic coexistence state of  \eqref{1.1} if, and only if,
\begin{equation}
\label{3.3}
\left\{
\begin{array}{ll}
n=u_0 + u_1 + \cdots + u_{n-1}, \\[2pt]
n=v_0 + v_1 + \cdots + v_{n-1},
\end{array}
\right.
\end{equation}
where we are using that $u_n=u_0$.  According to \eqref{3.2}, \eqref{3.3} can be equivalently expressed as
\begin{equation}
\label{3.4}
\left\{
\begin{array}{ll}
n=u_0\left[ 1 + e^{(1-v_0)A} + e^{(2-v_0-v_1)A}  + \cdots  + e^{(n-1 - v_0 - v_1 - \cdots - v_{n-2})A}\right], \\[4pt]
n=v_0\left[ 1 + e^{(u_1-1)B} + e^{(u_1+u_2-2)B}  + \cdots + e^{[u_1+u_2+\cdots +u_{n-1} - (n-1)]B}\right].
\end{array}
\right.
\end{equation}
As already shown in the proof of Theorem \ref{th2.1}, in the special case when $n=2$, $u_0$ can be easily obtained as a (explicit) function of $v_0$, which allowed as to express the system as a single equation of the unknown $x=v_0$. As this strategy does not work
when $n\geq 3$, in order to construct the set of $nT$-periodic solutions of \eqref{1.1} for all $n\geq 3$, throughout the rest of this paper we will make the additional assumption that
\begin{equation}
\label{3.5}
  x:=u_0=v_0 \quad \hbox{and} \quad A=B.
\end{equation}
Later, we will analyze their global topological structure through the distribution of their bifurcation points from the trivial curve $x=1$. Under these assumptions the next result holds. It is a pivotal result to express the Poincar\'e maps in a manageable way.

\begin{Lem}
\label{le3.1}
Suppose \eqref{3.3} and \eqref{3.5}. Then, for every $n\geq 2$,
\begin{equation}
\label{3.6}
   u_h=v_{n-h} \qquad \hbox{for all}\;\; h\in\{1,\ldots,n-1\}.
\end{equation}
Thus, the two equations of the system \eqref{3.4} coincide.
\end{Lem}
\begin{proof}
Fix $n\geq 2$. Then, owing to \eqref{3.2}, \eqref{3.3} and \eqref{3.5}, we find that
\[
  u_1 = u_0e^{(1-v_0)A} = v_0e^{(1-u_0)B} = v_0e^{[u_1+u_2+\cdots +u_{n-1}-(n-1)]B} =v_{n-1}.
\]
This relation provides us with the first identity of \eqref{3.6} ($h=1$). In particular, it shows \eqref{3.6} when $n=2$. More generally, suppose that $n\geq 3$ and that there exists $k\geq 1$ such that
\begin{equation}
\label{3.7}
  u_h=v_{n-h} \qquad \hbox{for all}\;\; h \in \{1,\ldots,n-k-1\}.
\end{equation}
Then, thanks to \eqref{3.2},  we have that
\[
 u_{n-k}=u_0e^{[(n-k)-v_0-v_1-\cdots -v_{n-k-1}]A}.
\]
Thus, by \eqref{3.5} and \eqref{3.7},
\begin{align*}
  u_{n-k} & =v_0e^{[(n-k)-u_0-u_{n-1}-\cdots -u_{k+1}]B} \\ & =
  v_0e^{[(n-k)-u_0-u_{n-1}-\cdots -u_{k+1}-u_k-u_{k-1}-\cdots-u_1+k+u_1+u_2+\cdots+u_k-k]B}.
\end{align*}
Thus, due to \eqref{3.3}, it becomes apparent that
\[
   u_{n-k}=v_0e^{(u_1+\cdots+u_k-k)B}=v_k,
\]
which concludes the proof of \eqref{3.6}. Therefore since \eqref{3.3} is equivalent to \eqref{3.4}, the two equtions of \eqref{3.4} coincide. The proof is complete.
\end{proof}

According to Lemma \ref{le3.1}, under condition \eqref{3.5}, to construct the fixed points of the Poincar\'e map $\mc{P}_n$, it suffices to consider any of the identities of \eqref{3.3} (or \eqref{3.4}), for instance, the first one. Thus, setting
\begin{equation}
\label{3.8}
  \v_n(u_0):=u_0+u_1(u_0)+u_2(u_0)+\cdots+u_{n-1}(u_0)-n, \qquad u_0>0,
\end{equation}
it becomes apparent that the zeros of $\v_n$ provide us with the positive fixed points of the Poincar\'e map $\mc{P}_n$.
By \eqref{3.2}
\[
  u_1(u_0):= u_0e^{(1-u_0)A},
\]
\begin{equation}
\label{3.9}
  u_2(u_0):=u_0e^{[2-u_0-v_1(u_0)]A}=u_0 e^{[2-u_0-v_0e^{(u_1-1)A}]A}=u_0 e^{(2-u_0-u_0e^{[u_0e^{(1-u_0)A}-1]A})A}
\end{equation}
and so on... though, in order to get a manageable  expression for $\v_n(u_0)$, all these terms should be reorganized in a slightly tricky way by using the relationships \eqref{3.3}, or \eqref{3.4}, which will be described in the proof of Theorem \ref{3.3}. The next result provides us with the Poincar\'e maps.

\begin{Prop}
\label{pr3.2}
Setting
\begin{equation}
\label{3.10}
\left\{
\begin{array}{lll}
E_0(x):=1,\quad E_1(x):=e^{(1-x)A},  \\[4pt]
E_n(x):=\left\{ \begin{array}{ll} e^{[x(E_1(x)+E_3(x)+\cdots+E_{n-1}(x))-\frac{n}{2}]A} & \quad \mathrm{if\ } \;\; n\in 2\N, \\[4pt]
e^{[\frac{n+1}{2}-x(E_0(x)+E_2(x)+\cdots+E_{n-1}(x))]A} & \quad \mathrm{if\ } \;\; n\in 2\N+1, \end{array} \right.
\end{array}
\right.
\end{equation}
for every $ n\geq 1 $ the Poincar\'e map is given through
\begin{equation}
\label{3.11}
	(u_n,v_n)=\mc{P}_n(x,x)=\left(xE_{2n-1}(x),xE_{2n}(x)\right).
\end{equation}
\end{Prop}
\begin{proof} By \eqref{3.2} and the definition of $E_1$ and $E_2$, it is easily seen that
\[
u_1=xe^{(1-x)A}=xE_1(x)\quad\hbox{and} \quad v_1=xe^{[xe^{(1-x)A}-1]A}=xE_2(x).
\]
Assume, as an induction hypothesis, that, for some integer $n\geq 1$,
\begin{equation}
\label{3.12}
u_{n-1}=xE_{2(n-1)-1}(x)\quad\hbox{and} \quad v_{n-1}=xE_{2(n-1)}(x).
\end{equation}
To prove \eqref{3.11}we argue as follows. According to \eqref{3.2},
\[
 u_n =xe^{(n-v_0-v_1-\cdots -v_{n-1})A}=u_{n-1}e^{(1-v_{n-1})A}.
\]
Thus, by the induction hypothesis and \eqref{3.10},
\begin{align*}
 u_n& = xE_{2n-3}(x)e^{(1-xE_{2n-2}(x))A}\\&=xe^{[\frac{2n-2}{2}-x(E_0+E_2+\cdots +E_{2n-4})]A}e^{(1-xE_{2n-2})A}
 \\&=xe^{[n-x(E_0+E_2+\cdots +E_{2n-4}+E_{2n-2})]A}=xE_{2n-1}(x).
\end{align*}
This provides us with the value of $u_n$ in \eqref{3.11}. Similarly,
\[
  v_n=xe^{(u_1+\cdots +u_n-n)A}=v_{n-1}e^{(u_n-1)A}.
\]
Thus, by \eqref{3.12}, since we already know that $u_n=xE_{2n-1}(x)$, we can infer that
\begin{align*}
v_n & =xE_{2n-2}(x)e^{(xE_{2n-1}(x)-1)A}
\\&=xe^{[x(E_1+E_3+\cdots +E_{2n-3})-(n-1)]A}e^{(xE_{2n-1}-1)A}
\\&=xe^{[x(E_1+E_3+\cdots +E_{2n-3}+E_{2n-1})-n]A}=xE_{2n}(x).
\end{align*}
This ends the proof.
\end{proof}

As a direct consequence,  from Proposition \eqref{pr3.2} one can get the auxiliary maps $\v_n$, $n\geq 1$, introduced in \eqref{3.8}.

\begin{Thm}
\label{th3.3} For every integer $n\geq 1$,
\begin{equation}
\label{3.13}
	\varphi_n(x)= \varphi_{n-1}(x) - 1 + xE_{n-1}(x).
\end{equation}	
\end{Thm}
\begin{proof}  First note that when $n$ is an odd integer, according to Lemma \ref{le3.1}, we have that
\begin{align}
\label{3.14}
  \varphi_n(x)& =u_0+u_1+\cdots +u_{\frac{n-1}{2}}+u_{\frac{n-1}{2}+1}+u_{\frac{n-1}{2}+2}+\cdots + u_{n-2}+u_{n-1}-n \nonumber \\ & =
u_0+u_1+\cdots +u_{\frac{n-1}{2}}+v_{\frac{n-1}{2}}+v_{\frac{n-1}{2}-1}+\cdots + v_{2}+v_{1}-n\nonumber \\ & =
u_0+u_1+v_1+u_2+v_2+\cdots + u_{\frac{n-1}{2}-1}+v_{\frac{n-1}{2}-1}+u_{\frac{n-1}{2}}+v_{\frac{n-1}{2}}-n.
\end{align}
Similarly, when $n$ is even,
\begin{align}
\label{3.15}
  \varphi_n(x)& =u_0+u_1+\cdots +u_{\frac{n}{2}}+u_{\frac{n}{2}+1}+u_{\frac{n}{2}+2}+\cdots + u_{n-2}+u_{n-1}-n \nonumber \\ & =
u_0+u_1+\cdots +u_{\frac{n}{2}}+v_{\frac{n}{2}-1}+v_{\frac{n}{2}-2}+\cdots + v_{2}+v_{1}-n\nonumber \\ & =
u_0+u_1+v_1+u_2+v_2+\cdots + u_{\frac{n}{2}-1}+v_{\frac{n}{2}-1}+u_{\frac{n}{2}}-n.
\end{align}
To prove \eqref{3.13} a complete induction argument will be used. When $n=1$,
\[
  \v_1(x)=x-1.
\]
When $n=2$, by \eqref{3.2},
\[
  \v_2(x):=x+u_1-2=x+x e^{(1-x)A}-2=\v_1(x)-1+x E_1(x).
\]
As the complete induction hypothesis,  suppose that, for any given $\nu\geq 2$, \eqref{3.13} holds for every
$n \in \{1,2,\ldots,2\nu-3,2\nu-2\}$. Then, thanks to \eqref{3.14} and \eqref{3.15},
\begin{align*}
  \v_{2\nu-1}(x) & = u_0+u_1+v_1+u_2+v_2+\cdots + u_{\nu-2}+v_{\nu-2}+u_{\nu-1}+v_{\nu-1}-2\nu+1, \\
  \v_{2\nu}(x)& = u_0+u_1+v_1+u_2+v_2+\cdots + u_{\nu-1}+v_{\nu-1}+u_{\nu}-2\nu.
\end{align*}
Thus, thanks to \eqref{3.11},
\begin{equation}
\label{3.16}
\begin{split}
\varphi_{2\nu-1}(x)=x & +xE_1(x) +x E_2(x) +x E_3(x) + x E_4(x)+\cdots  \\ & + xE_{2\nu-5}(x)+xE_{2\nu-4}(x) +
x E_{2\nu -3}(x)+ x E_{2\nu-2}(x) -2\nu +1.
\end{split}
\end{equation}
Similarly,
\begin{equation}
\label{3.17}
\begin{split}
\varphi_{2\nu}(x)=x & +xE_1(x) +x E_2(x) +x E_3(x) + x E_4(x)+\cdots  \\ & + xE_{2\nu-3}(x)+xE_{2\nu-2}(x) +
x E_{2\nu -1}(x)-2\nu.
\end{split}
\end{equation}
Therefore,  by the induction hypothesis,
\[
\varphi_{2\nu-1}(x)=\v_{2\nu-2}(x) -1 + x E_{2\nu-2}(x).
\]
Similarly,
\[
  \varphi_{2\nu}(x)= \v_{2\nu-1}(x)-1+x E_{2\nu-1}(x).
\]
The proof is complete.
\end{proof}

\begin{Rem}
\label{re3.4} By \eqref{3.16} and \eqref{3.17} it becomes apparent that
\[
  \v_n(0)=-n <0 \quad\hbox{and}\quad \v_{n}(n) > 0 \quad \hbox{for all integer}\;\; n\geq 2.
\]
\end{Rem}

According to Theorem \ref{th3.3}, it is easily seen that
\begin{align}
\label{3.18}
  \v_1(x) & =x-1, \\ \v_2(x) & = \v_1(x)-1+x e^{(1-x)A},\nonumber \\
  \v_3(x)  &=  \v_2(x)-1 +xe^{(xe^{(1-x)A}-1)A},\nonumber\\
  \varphi_4(x) & =\v_3(x)-1+ xe^{(2-x-xe^{(xe^{(1-x)A}-1)A})A}.\nonumber
\end{align}
Crucially, in the formula for $\v_2(x)$ given by Theorem \ref{th3.3} it is only required to compose two exponentials, while in
\eqref{3.9} we had to nest three. Such reduction in the complexity of $\v_2$ is explained by the  symmetries revealed by
Lemma \ref{le3.1} which facilitated the reorganization of the terms of $\v_n$ as to get a function with a minimal number of nested exponentials, much like in the algorithm of the proof of Theorem
\ref{th3.3}. Using this algorithm, the number of nested exponentials decreases by one when $n$ is odd, and
each of the the $E_n$'s defined by \eqref{3.10} consists of a composition of exactly $n$ exponentials.
The relevance of this reduction will not be completely understood until the next sections, where  the structure
of the zeros of the $\v_n$'s introduced in \eqref{3.8} will be analyzed. Those zeros are  the positive fixed points of the $nT$-time Poincar\'e maps.
\par
The main technical difficulty to determine the zeros of the  $\v_n$'s, even from the point of view of numerical analysis, relies on the high sensitivity of these functions to very small variations in the value of the parameter $A=B$. The higher the number of exponentials nested, the higher the sensitivity in $A$. As a result, when one tries to determine numerically the zeros of the map $\v_4$ for values of $A$ near $4$, the function $\v_4(x)$ takes values of order $10^{31}$ in a  neighborhood of zero. So, there is no chance to compute
the zeros of these maps assisted by the computer. When dealing with $\v_5$ the value of the parameter $A=B$ should not exceed the value
$2.5$, which is extremely unsatisfactory for our purposes here. These technical troubles inherent to the internal structure of the
associated maps $\v_n$ push us to make a direct analysis of the global structure of their zeros. In order to perform this global analysis we first need to ascertain the set of bifurcation points of $\v_n=0$ from the curve $(A,1)$. This analysis will be carried out in the next section.

\section{A canonical chain of associated polynomials}
Searching  for the potential bifurcation points from the curve $(A,1)$ to $nT$-periodic coexistence states, this section analyzes the spectrum of the linearized family
\[
\mathfrak{L}(n;A):= \frac{d\v_n (A,1)}{dx},\quad n\in\N,
\]
i.e., its zero set as a function of the parameter $A$, as well as the global structure of $\mf{L}(n;A)$.
Note that, since $(A,1)$ is the $T$-periodic coexistence state, it also provides us with a $nT$-periodic solution for all
$n\geq 1$ and, hence, by construction,  $\v_n(A,1)=0$ for all $A>0$ and $n\geq 1$. The curve $(A,1)$, $A>0$,  is the \emph{trivial} curve, as it is known. It is the curve from which are going to bifurcate the $nT$-periodic coexistence states of \eqref{1.1} under assumption \eqref{3.5}. Note also that, since every  $nT$-periodic solution is $knT$-periodic for all integer $k\geq 1$,
\begin{equation}
\label{4.1}
  \v_{kn}(x)=0 \quad \hbox{for all}\;\; x \in \v_n^{-1}(0)\quad \hbox{and}\quad k\geq 1.
\end{equation}
Throughout the rest of this paper we will denote
\begin{equation}
\label{4.2}
  p_n(A):=  \frac{d\v_n (A,1)}{dx}=\mf{L}(n;A),\qquad A>0.
\end{equation}
Differentiating with respect to $x$ the identity \eqref{3.13} yields
\begin{equation}
\label{4.3}
p_n(A)=p_{n-1}(A)+ E_{n-1}(1)+E'_{n-1}(1)\qquad \hbox{for all}\;\; A>0.
\end{equation}
The next result shows that $p_n \in\Z \left[ A\right]$.

\begin{Lem}
\label{le4.1}
For every $ n\in\N $, $p_n(A) $ is a polynomial in the variable $A$ with integer coefficients, i.e., $ p_n \in\Z[A] $.
\end{Lem}
\begin{proof}
By \eqref{3.11}, it becomes apparent that, since $(1,1)$ is a fixed point of $ \mc{P}_n $,
\begin{equation*}
	(1,1)=\mc{P}_n(1,1)=\left(E_{2n-1}(1),E_{2n}(1)\right)
\end{equation*}
for all integer $n\geq 1$. Thus,
\begin{equation}
\label{4.4}
  E_n(1)=1 \quad \hbox{for all}\;\; n\geq 0.
\end{equation}
Thus, \eqref{4.3} becomes
\begin{equation}
\label{4.5}
p_n(A)=p_{n-1}(A)+ 1 +E'_{n-1}(1)
\end{equation}
for all $A>0$ and $n\geq 1$. Therefore, due to \eqref{3.18}  ,  $p_1(A) = \frac{d\v_1(A,1)}{dx}= 1$ and iterating \eqref{4.5} $n-2$ times show that,
for every integer $n\geq 2$,
\begin{equation}
\label{4.6}
  p_n(A)= n+E_1'(1)+E_2'(1)+\cdots +E_{n-1}'(1).
\end{equation}
Consequently, to complete the proof it suffices to show that $ E_n'(1)\in\Z[A] $ for all $n\geq 1$. Indeed, by \eqref{3.10},
$E_0'(1)=0$, $ E_1'(1)=-A $ and
\begin{equation}
\label{4.7}
E_n'(1)=\left\{ \begin{array}{ll}\displaystyle{ A  \sum_{\substack {j=1\\j\in 2\N+1}}^{n-1} \left[
  E_{j}(1) +E_{j}'(1)\right] } &\qquad \hbox{if\ } n\in2\N,\\[14pt]
\displaystyle{ -A \sum_{\substack {j=0\\ j\in 2\N}}^{n-1}\left[ E_{j}(1)+ E_{j}'(1) \right]} & \qquad \hbox{if\ } n\in2\N+1.
\end{array}
\right.
\end{equation}
Thus, by a complete induction argument it becomes apparent that  $ E_n'(1)\in\Z[A] $ for all $ n\in \N $. This concludes the proof.
\end{proof}

\begin{Rem}
\label{re4.2} In Section 5 we will prove that all the roots of the polynomial $p_n(A)$ are simple. In other words,
\[
  p'_n(r) = \frac{d \mf{L}}{d A}(n;r)\neq 0
\]
for all $r\in p_n^{-1}(0)$. Thus, the transversality condition of M. G. Crandall and P. H. Rabinowitz \cite{CR71} holds true.
Therefore, by the main theorem of \cite{CR71}, at every positive root of $p_n(A)$, $r$, an analytic curve of $nT$-periodic coexistence states of \eqref{1.1} bifurcates from $(A,1)$ at $r$. This feature explains our interest here in analyzing the nature and the distribution of the positive roots of the polynomials $p_n(A)$, $n\in\N$.
\end{Rem}

\begin{Rem}
\label{re4.3}
Occasionally, we will make explicit the dependence of the function $\v_n(x)$ on the parameter $A$ by setting
$\v_n(A,x)$, instead of $\v_n(x)$. Similarly, we will set  $E_n(A,x):=E_n(x)$ for all $n\in\N$. According to
\eqref{3.10}, $E_n(0,x)=1$ for all $n\in\N$ and $x \in [0,n]$. Thus, \eqref{3.13} yields
\begin{equation*}
	\varphi_n(0,x)= \varphi_{n-1}(0,x) - 1 + x
\end{equation*}	
for all $n\in\N$ and $x\in [0,n]$. Therefore, iterating $n-1$ times, it becomes apparent that
\begin{equation}
\label{4.8}
	\varphi_n(0,x)=n(x-1)\quad\hbox{for all}\;\; n\in\mathbb{N}.	
\end{equation}
As the zeros of $\v_n(A,x)$ provide us with the $nT$-periodic positive solutions of \eqref{1.1}, it follows from \eqref{4.8} that
$x=1$ is the unique  $nT$-periodic solution, for all $n\in \N$,  at the particular value of the parameter $A=0$.
\end{Rem}

The next list collects the polynomials $p_n(A)$ for $1\leq n\leq 13$.

\begin{small}
\begin{align*}
p_1(A)&=1\\
p_2(A)&=-A+2\\
p_3(A)&=-A^2+3\\
p_4(A)&=A^3-2A^2-2A+4\\
p_5(A)&=A^4-5A^2+5\\
p_6(A)&=-A^5+2A^4+4A^3-8A^2-3A+6\\
p_7(A)&=-A^6+7A^4-14A^2+7\\
p_8(A)&=A^7-2A^6-6A^5+12A^4+10A^3-20A^2-4A+8\\
p_9(A)&=A^8-9A^6+27A^4-30A^2+9\\
p_{10}(A)&=-A^9+2A^8+8A^7-16A^6-21A^5+42A^4+20A^3-40A^2-5A+10\\
p_{11}(A)&=-A^{10}+11A^8-44A^6+77A^4-55A^2+11\\
p_{12}(A)&=A^{11}-2A^{10}-10A^9+20A^8+36A^7-72A^6-56A^5+112A^4+35A^3-70A^2-6A+12\\
p_{13}(A)&=A^{12}-13A^{10}+65A^8-156A^6+182A^4-91A^2+13.
\end{align*}
\end{small}
The next table collects the coefficients of all the polynomials listed above.
\begin{table}[htb]
	\centering
	\begin{tabular}{|r|r|r|r|r|r|r|r|r|r|r|r|r|}
		    \hline
 $A^{12}$ & $ A^{11} $&$ A^{10} $&$ A^9 $&$ A^8 $&$ A^7 $&$ A^6 $&$ A^5 $&$ A^4 $&$ A^3 $&$ A^2 $&$ A^1 $&$ A^0 $  \\ \hline
		     0 & 0 & 0 & 0  &0  &0 & 0 & 0 & 0 & 0 & 0 & 0  & 1
			\\ \hline  0 &0 &0  &0  &0 &0 &0   &0   &0  &0  &0   &-1 &2
			\\ \hline 0 &0 &0  &0  &0 &0 &0   &0   &0  &0  &-1  &0  &3
			\\ \hline 0 &0 &0  &0  &0 &0 &0   &0   &0  &1  &-2  &-2 &4
			\\ \hline 0 &0 &0  &0  &0 &0 &0   &0   &1  &0  &-5  &0  &5
			\\ \hline 0 &0 &0  &0  &0 &0 &0   &-1  &2  &4  &-8  &-3 &6
			\\ \hline 0 &0 &0  &0  &0 &0 &-1  &0   &7  &0  &-14 &0  &7
			\\ \hline 0 &0 &0  &0  &0 &1 &-2  &-6  &12 &10 &-20 &-4 &8
			\\ \hline 0 &0 &0  &0  &1 &0 &-9  &0   &27 &0  &-30 &0  &9
			\\ \hline 0 &0 &0  &-1 &2 &8 &-16 &-21 &42 &20 &-40 &-5 &10
			\\ \hline 0 &0 &-1 &0  &11&0 &-44 &0   &77 &0  &-55 &0  &11
			\\ \hline 0 &1 &-2 &-10&20&36&-72 &-56 &112&35 &-70 &-6 &12
			\\ \hline 1 &0 &-13&0  &65&0 &-156&0   &182&0  &-91 &0  &13
			\\\hline
	\end{tabular}
\caption{First thirteen polynomials coefficients.}
\label{Tab1}
\end{table}

By simply having a glance to these polynomials, it becomes apparent that the following properties hold:
\begin{enumerate}
\item[(a)] The constant terms of $p_n(A)$ equals $n$.
\item[(b)] The degree of $p_n(A)$ equals $n-1$.
\item[(c)] The leading coefficients of $p_{4n}(A)$ and $p_{4n+1}(A)$ equal $1$, while the
leading coefficients of $p_{4n+2}(A)$ and $p_{4n+3}(A)$ equal $-1$.
\item[(d)] $p_{2n}(2)=0$ for all integer $n\geq 1$. Thus, $p_2|p_{2n}$ for all $n\geq 1$.
\item[(e)] $p_{2n+1}(A)$ is an even function.
\end{enumerate}
Besides these properties, it seems all the coefficients of $p_n(A)$, except the leading one, must be multiples of $n$ if $n$ is a prime integer, though this property will not be used in this paper. The next result shows the property (a).

\begin{Lem}
\label{le4.4}
$ p_n(0)=n $ for all $ n\geq 1$.
\end{Lem}
\begin{proof}
By \eqref{4.7}, $ \frac{dE_n(0,1)}{dx}=0 $. Hence, due to \eqref{4.6}, $p_n(0)=n$ for all $ n\geq 1$.
\end{proof}

The next result establishes the properties (b) and (c).

\begin{Lem}
\label{le4.5}
For every integer $n\geq 1$, $ {\rm{deg}}(p_n)=n-1 $. Moreover, the leading coefficients of $ p_n $ equal $ 1 $ if $ n\in4\N\cup (4\N+1) $ and $ -1 $ if $ n\in (4\N+2)\cup (4\N+3) $.
\end{Lem}
\begin{proof} By the proof of Lemma \ref{le4.1}, we already know that
\[
  E'_n(A,1):=\frac{d E_n}{dx} (A,1)
\]
is a polynomial in $A$ for all integer $n\geq 1$. Next, we will show that it has degree $ n $. To prove it, a complete induction argument will be used. According to \eqref{3.10}, we already know that
\[		
   {\rm{deg}}(E_0'(A,1))={\rm{deg}}(0)=0  \quad\hbox{and}\quad  {\rm{deg}}(E_1'(A,1))={\rm{deg}}(-A)=1.
\]
As the induction assumption, assume that
\[
   {\rm{deg}}(E_j'(A,1))=j \quad \hbox{for all}\; \; j<n.
\]
Then, owing to \eqref{4.7},  it follows that
\begin{equation}
\label{4.9}
  {\rm{deg}}(E_n'(A,1))=n,\qquad n\geq 0.
\end{equation}
Therefore, by \eqref{4.6},
\[
   {\rm{deg}}(p_n)=n-1.
\]
Subsequently, for any given polynomial, $q \in \Z[A]$, we will denote by $\ell(q)$ the leading coefficient
of $q(A)$. According to Table \ref{Tab1}, we already know that
\[
 \ell(p_5)=1.
\]
As an induction hypothesis, assume that
\begin{equation}
\label{4.10}
  \ell(p_{4(n-1)+1})=1.
\end{equation}
By \eqref{4.6}, \eqref{4.7} and \eqref{4.9}
\begin{align*}
\ell(p_{4n-2}) & = {\color{blue} \ell(E_{4n-3}'(A,1))}=-\ell(E'_{4(n-1)}(A,1))=-\ell(p_{4(n-1)+1}),\\
\ell({p_{4n-1}}) & ={\color{brown} \ell({E_{4n-2}'(A,1)})} ={\color{blue} \ell({E_{4n-3}'(A,1)})}=-\ell(p_{4(n-1)+1}),\\
\ell({p_{4n}}) & ={\color{cyan} \ell({E_{4n-1}'(A,1)})}=-{\color{brown} \ell({E_{4n-2}'(A,1)})}=\ell(p_{4(n-1)+1}),\\
\ell({p_{4n+1}}) & =\ell({E_{4n}'(A,1)})={\color{cyan} \ell({E_{4n-1}'(A,1)})}=\ell(p_{4(n-1)+1}).\\
\end{align*}
By \eqref{4.10}, the proof is complete.
\end{proof}

As a consequence of these lemmas, the next result holds.

\begin{Prop}
\label{pr4.6}
Suppose \eqref{3.5}. Then, the problem \eqref{1.1} possesses infinitely many subharmonics. In other words, there exists a sequence of integers $\{n_m\}_{m\geq 1}$ with
\[
  \lim_{m\to +\infty} n_m =+\infty,
\]
such that \eqref{1.1} has at least a $n_mT$-periodic coexistence state for every $m\geq 1$.
\end{Prop}
\begin{proof}
Since $p_n(0)=n$ for all $n\in\N$ and, thanks to Lemma \ref{le4.5}, for every integer $n\geq 1$,
\[
  \ell(p_{4n+2})=\ell(p_{4n+3})=-1,
\]
it becomes apparent that $p_{4n+2}(A)$ (resp. $p_{4n+3}(A)$) possesses a root, $A_{4n+2}$ (resp.
$A_{4n+3}$), where it changes of sign. Thus, for every integer $n\geq 1$, there exist two odd integers, $i_n, j_n\geq 1$, for which
\begin{align*}
  p^{k)}_{4n+2}(A_{4n+2}) & =0, \quad 0\leq k \leq i_n-1, \quad  p^{i_n)}_{4n+2}(A_{4n+2}) \neq 0, \\
   p^{k)}_{4n+3}(A_{4n+3}) & =0, \quad 0\leq k \leq j_n-1, \quad  p^{j_n)}_{4n+3}(A_{4n+3}) \neq 0.
\end{align*}
Thus, the algebraic multiplicity of \cite{ELG87} for these polynomials at those roots
is given by
\[
  \chi[p_{4n+2}(A);A_{4n+2}]= i_n, \qquad  \chi[p_{4n+3}(A);A_{4n+3}]= j_n.
\]
As these integers are odd, by Theorem 5.6.2 of \cite{LG01}, the local topological indexes of $p_{4n+2}(A)$ and $p_{4n+3}(A)$ change as $A$ crosses $A_{4n+2}$ and $A_{4n+3}$, respectively. Therefore, by Theorem
6.2.1 of \cite{LG01}, there exist two components of $(4n+2)T$-periodic solutions and $(4n+3)T$-periodic
solutions bifurcating from the trivial solution $(A,1)$ at the roots $A_{4n+2}$ and $A_{4n+3}$, respectively. This ends the proof.
\end{proof}

The next result establishes Property (d).

\begin{Lem}
\label{le4.7}
$ p_2|p_{2n} $ for all $ n\geq 1$. Thus, since $p_2(A)=-A+2$, $r=2$ is a root of $p_{2n}(A)$
for all integer $n\geq 1$.
\end{Lem}
\begin{proof} By \eqref{4.1}, any $2T$-periodic solution is a $2nT$-periodic solution for all
$n\geq 1$. Thus, any bifurcation point from $(A,1)$ to $2T$-periodic solutions must be a
bifurcation point to $2nT$-periodic solutions. Since the unique bifurcation value to $2T$-periodic
solutions is the root of $p_2(A)=-A+2$, given by $r=2$, it becomes apparent that $p_{2n}(2)=0$
for all integer $n\geq 1$. Therefore, $p_2|p_{2n}$ for all $n\geq 1$. This ends the proof.
\end{proof}

The next list of polynomials, collecting $p_{2n+1}(A)$ and $\frac{p_{2n}(A)}{2-A}$, for $1 \leq n \leq 6$,
might be helpful to understand the (very sharp) identity established by the next result.

\begin{align*}
	\dfrac{p_2(A)}{2-A}&=1\\
	p_3(A)&=-A^2+3\\
	\dfrac{p_4(A)}{2-A}&=-A^2+2\\
	p_5(A)&=A^4-5A^2+5\\
	\dfrac{p_6(A)}{2-A}&=A^4-4A^2+3\\
	p_7(A)&=-A^6+7A^4-14A^2+7\\
	\dfrac{p_8(A)}{2-A}&=-A^6+6A^4-10A^2+4\\
	p_9(A)&=A^8-9A^6+27A^4-30A^2+9\\
	\dfrac{p_{10}(A)}{2-A}&=A^8-8A^6+21A^4-20A^2+5\\
	p_{11}(A)&=-A^{10}+11A^8-44A^6+77A^4-55A^2+11\\
	\dfrac{p_{12}(A)}{2-A}&=-A^{10}+10A^8-36A^6+56A^4-35A^2+6\\
	p_{13}(A)&=A^{12}-13A^{10}+65A^8-156A^6+182A^4-91A^2+13.
\end{align*}
\vspace{0.2cm}

\begin{Thm}
\label{th4.8} The following identity holds
\[
  \dfrac{p_n(A)}{2-A}=p_{n-1}(A) - \dfrac{p_{n-2}(A)}{2-A}
\]
for all $n\in2\N$, whereas
\[
  \dfrac{p_n(A)}{2+A}=p_{n-1}(A) - \dfrac{p_{n-2}(A)}{2+A}
\]
for all $n\in 2\N+1$.
\end{Thm}
\begin{proof}
First, we will prove the next relationships
\begin{equation}
\label{4.11}
		\left\{
		\begin{array}{ll}
		\displaystyle{- \frac{n}{2}-1-\sum_{\substack{j=1\\ j \in 2\N}}^{n} E_{j}'(A,1) =1-A+\sum_{j=3}^{n+1}(-1)^jp_j,}&\quad  n\in 2\N,\\[11pt]
    	\displaystyle{ \left[\frac{n}{2}\right]+1+\sum_{\substack{j=1\\ j \in 2\N+1}}^{n} E_{j}'(A,1) =1-A+\sum_{j=3}^{n+1}(-1)^jp_j,}&\quad  n\in 2\N+1.
    	\end{array}
    	\right.
\end{equation}
Since $p_2(A)=2-A$, particularizing \eqref{4.5} at $n=3$ yields
\[
    	-2-E_2'(A,1)=1-A-p_3(A),
\]
which is \eqref{4.11} for $n=2$. As the induction assumption, assume that \eqref{4.11} holds
for some $n=2m$ with $m\geq 1$, i.e.,
\begin{equation}
\label{4.12}
- m -1-\sum_{\substack{j=1\\ j \in 2\N}}^{2m} E_{j}'(A,1) =1-A+\sum_{j=3}^{2m+1}(-1)^jp_j(A).
\end{equation}
According to \eqref{4.6},
\begin{equation}
\label{4.13}
    	2m+2+E_1'(A,1)+E_2'(A,1)+\cdots+E_{2m}'(A,1)+E_{2m+1}'(A,1)=p_{2m+2}(A).
\end{equation}
Thus, adding \eqref{4.12} and \eqref{4.13}, we obtain that
\begin{equation}
\label{4.14}
   m+1 + \sum_{\substack{j=1\\ j \in 2\N+1}}^{2m+1} E_{j}'(A,1)=1-A+\sum_{j=3}^{2m+2}(-1)^jp_j(A).
\end{equation}
Equivalently,
\begin{equation*}
\left[\frac{2m+1}{2}\right]+1 + \sum_{\substack{j=1\\ j \in 2\N+1}}^{2m+1} E_{j}'(A,1)=1-A+\sum_{j=3}^{2m+2}(-1)^jp_j(A),
\end{equation*}
which shows the validity of \eqref{4.11} for $n=2m+1$. To prove the validity of \eqref{4.11}
for $n=2(m+1)=2m+2$, we can argue similarly. Again by \eqref{4.6},
\begin{equation}
\label{4.15}
    	2m+3+E_1'(A,1)+E_2'(A,1)+\cdots+E_{2m+1}'(A,1)+E_{2m+2}'(A,1)=p_{2m+3}(A).
\end{equation}
Hence, subtracting \eqref{4.15} from \eqref{4.14} yields
\begin{equation}
\label{4.16}
  -m-2 - \sum_{\substack{j=1\\ j \in 2\N}}^{2m+2} E_{j}'(A,1)=1-A+\sum_{j=3}^{2m+3}(-1)^jp_j(A).
\end{equation}
Since
\[
  -\frac{2m+2}{2}-1 =-m-2,
\]
\eqref{4.16} provides us with \eqref{4.11} for $n=2m+2$, which ends the proof of \eqref{4.11}.
\par
By \eqref{4.4}, it follows from \eqref{4.5} and \eqref{4.7} that
\begin{equation}
\label{4.17}
    	p_n(A)=\left\{
    	\begin{array}{ll}
    	\displaystyle{p_{n-1}(A)+1+A\Big[-\dfrac{n}{2}-\sum_{\substack{j=1\\j\in2\N}}^{n-2}E_{j}'(A,1)\Big]} & \quad \hbox{if\ } n\in2\N,
    	\\[15pt]
    	\displaystyle{p_{n-1}(A) + 1 + A\Big[\dfrac{n-1}{2}+\sum_{\substack{j=1\\j\in2\N+1}}^{n-2}E_{j}'(A,1)
    \Big]} &\quad \hbox{if\ } n \in 2\N+1.    	\end{array}
    	\right.
\end{equation}
On the other hand, when $n\in 2\N$, it follows from \eqref{4.11} and \eqref{4.17} that
\begin{equation*}
\begin{split}
 p_n(A)-p_{n-1}(A)&=1+A\Big[-\dfrac{n}{2}-\sum_{\substack{j=1\\j\in2\N}}^{n-2}E_{j}'(A,1) \Big]\nonumber\\ &=
 1+A\Big[-\dfrac{n-2}{2}-1-\sum_{\substack{j=1\\j\in2\N}}^{n-2}E_{j}'(A,1) \Big]\nonumber  \\ & =
 1+A\Big[ 1-A+\sum_{j=3}^{n-1}(-1)^jp_j\Big]\nonumber\\&=1-Ap_{n-1}(A)+A\Big[1-A+\sum_{j=3}^{n-2}(-1)^jp_j\Big]
 \nonumber\\&=-Ap_{n-1}(A)+1+A\Big[\displaystyle{\dfrac{n-2}{2}+\sum_{\substack{j=1\\ j\in 2\N+1}}^{n-3} E_{j}'(A,1)}\Big]\\&=-Ap_{n-1}(A)+p_{n-1}(A)-p_{n-2}(A).
\end{split}
\end{equation*}
Therefore, for every $n\in 2\N$,
\begin{equation*}
    	p_n(A)= (2-A)p_{n-1}(A)-p_{n-2}(A).
\end{equation*}
The proof is complete for $n$ even. Subsequently, we assume that $n$ is odd.  Arguing as in the previous case, from \eqref{4.17} and \eqref{4.11} the following chain of identities holds
\begin{align*}
p_n(A)-p_{n-1}(A)& = 1+A\Big[\dfrac{n-1}{2}+\sum_{\substack{j=1\\j\in2\N+1}}^{n-2}E_{j}'(A)\Big]\\ &=1+A\Big[1-A+\sum_{j=3}^{n-1}(-1)^jp_j\Big]\\
&=1+ Ap_{n-1}(A)+A\Big[1-A+\sum_{j=3}^{n-2}(-1)^jp_j\Big]\\
&=Ap_{n-1}(A)+1+ A\Big[-\dfrac{n-1}{2}-\sum_{\substack{j=1\\j\in2\N}}^{n-3}E_{j}'(1)\Big]
\\&=Ap_{n-1}(A)+p_{n-1}(A)-p_{n-2}(A).
\end{align*}
Therefore, for every $n\in 2\N+1$,
\[
		p_n(A) =(2+A)p_{n-1}(A)-p_{n-2}(A).
\]
This ends the proof.
\end{proof}
	
Theorem \ref{th4.8}  can be summarized into the next generalized identity
\begin{equation}
\label{4.18}
	p_n(A)=[2-(-1)^n A]p_{n-1}(A)-p_{n-2}(A),\qquad n\in\N.
\end{equation}
As a by-product of these identities, the next result, establishing Property (e) at the beginning of the section, holds.
\begin{Cor}
\label{co4.9}
For every $n\geq 1$, the polynomials $\dfrac{p_{2n}(A)}{2-A}$ and $p_{2n+1}(A)$ are even.
\end{Cor}
\begin{proof}
We already know that
\[
\frac{p_2(A)}{2-A}=1\qquad\hbox{and}\qquad p_3(A)=-A^2+3.
\]	
Arguing by induction, assume that  $\dfrac{p_{2m-2}(A)}{2-A}$  and $p_{2m-1}(A)$ are even polynomials for some $m\geq1$. Then, by \eqref{4.18},
\[
\frac{p_{2m}(A)}{2-A}=p_{2m-1}(A)-\frac{p_{2m-2}(A)}{2-A}
\]
must be also even, because it is sum of two even functions. Similarly, since  $p_{2m+1}$ can be expressed in the form
\[
p_{2m+1}(A)=(2+A)p_{2m}(A)-p_{2m-1}(A)=(4-A^2)\dfrac{p_{2m}(A)}{2-A}-p_{2m-1}(A),
\]
it becomes apparent that $p_{2m+1}(A)$ is also an even polynomial. The proof is completed.
\end{proof}

\section{Characterizing the bifurcation points from $(A,1)$}
	
The following definition will be used in the statement of the main theorem of this section.
	
\begin{Def}
\label{de5.1}
Given two arbitrary polynomials $q_1, q_2\in \Z[A]$, it is said that the roots of $q_1$ are separated by the roots of $q_2$ if all the roots of $q_2$ lye in between the maximal and minimal roots of $q_1$ and any pair of consecutive roots of $q_2$ contains exactly one root of $q_1$.
\end{Def}
	
The main theorem of this section can be stated as follows. It counts the number of roots of each of the polynomials $p_n(A)$, $n\geq 1$, establishing that there are as many roots as indicated by the degree, that all of them are real and algebraically simple and that the positive roots of $p_{n+1}(A)$ are always separated by the positive roots (less than $2$ if $n\in2\N$) of $p_n(A)$. So, it counts all roots establishing their relative positions.
	
\begin{Thm}
\label{th5.2}
For every $n\geq 2$, the positive roots of $p_{2n}(A)$ are separated by the positive roots of $p_{2n-1}(A)$, and the positive roots of $p_{2n+1}(A)$ are separated by those of $p_{2n}(A)$ less than $2$. Moreover, for every $n\geq 1$, the even polynomials  $\tfrac{p_{2n}(A)}{2-A}$ and $p_{2n-1}(A)$ have (exactly) $n-1$ positive roots. Thus, since they are even with degree $2n-2$, they must have another
$n-1$ negative roots and, therefore, all roots are real and simple.
\end{Thm}
\begin{proof} As we have already constructed the  associated polynomials above, it is easily seen that all the thesis of Theorem \ref{th5.2} hold to be true for $2\leq n\leq 6$. This task can be easily accomplished
by simply looking at Figure \ref{Fig2}, where  we have plotted all the positive roots of $p_n(A)$ for $2\leq n\leq 13$. These roots are located in the interval $(0,2]$ and have been represented in abscisas at different levels according to $n$. As inserting in the same interval $(0,2]$ all the zeros of the first $13$ polynomials would not be of any real  help for understanding their fine distribution, we have superimposed them at $13$ different levels, each of them containing the positive roots of each of the polynomials $p_n$, $2\leq n\leq 13$. In total we are representing $42$ roots, though some of them are common roots of
different polynomials as a result of the fact that any $kT$-periodic solution must be a $nkT$-periodic solution for all $n\geq 1$. These common roots have been represented in vertical dashed lines to emphasize that all roots on them share the same abscisa value. In such case, the ordinates provide us with the corresponding value of $n$. By simply having a glance at Figure \ref{Fig2}, it is easily realized how the two roots of the polynomial $p_4$ are separated by the root of $p_3$, the $3$ roots of $p_6$ are separated by the $2$ roots of $p_5$, the $4$ roots of $p_8$ are separated by the $3$ of $p_7$, and so on... Similarly, the two roots of $p_5$ are separated by the unique
root of $p_4$ different from $2$, the $3$ roots of $p_7$ are separated by the $2$ roots of $p_6$ different from $2$, and so on...  The proof of the theorem will be delivered in two steps by induction in both cases.  Since $\tfrac{p_2(A)}{2-A}=1$ does not admit any root, this is a very special case that will not play any
rol in these induction arguments.
\begin{figure}[h!]
	\centering
	\includegraphics[scale=0.65]{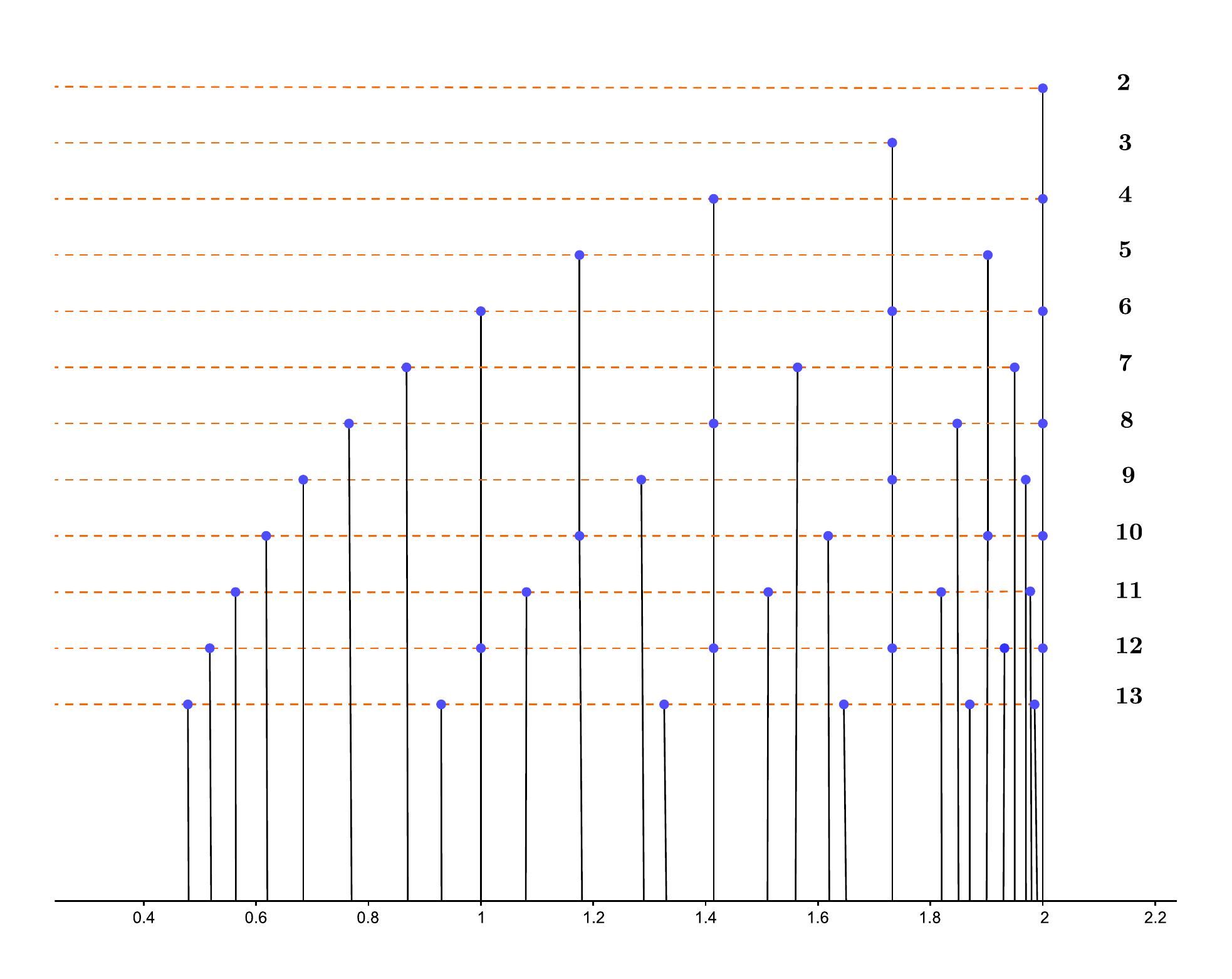}
	\caption{Positive roots of $ p_n $, $2\leq n \leq 13$. }
	\label{Fig2}
\end{figure}
\vspace{0.4cm}
\par
\noindent\textbf{Step 1: Passing from $p_{2n}(A)$ to $p_{2n+1}(A)$, $n\geq 2$.} According to Figure \ref{Fig2}, it becomes apparent that the two positive roots of $p_4(A)$ are separated by the unique root of $p_3(A)$. Moreover, all these zeros are real and simple and each of the polynomials
\[
  p_{3}(A)=-A^2+3,\qquad \frac{p_4(A)}{2-A}= -A^2+2,
\]
has a unique positive root. Arguing by induction, assume that $p_{2n-1}(A)$ and $p_{2n}(A)$ satisfy all the assertions of the statement of the theorem  for some $n\geq 2$. In other words, all the positive roots of these  polynomials are real and algebraically simple, the positive roots of $p_{2n}(A)$ are separated by the positive roots of $p_{2n-1}(A)$, and the polynomials  $\tfrac{p_{2n}(A)}{2-A}$ and $p_{2n-1}(A)$ have (exactly) $n-1$ positive roots. We claim that the positive roots of the polynomial $p_{2n+1}(A)$ are real and simple, that they are separated by the positive roots of $p_{2n}(A)$, except for $2$, and that it has (exactly) $n$ positive roots. Indeed, by Theorem \ref{th4.8} , we already know that
\begin{equation}
\label{5.1}
  p_{2n+1}(A)=(2+A)p_{2n}(A) - p_{2n-1}(A).
\end{equation}
First, we will show the previous claim in the case when $ 2n\in 4\N+2 $. So, suppose $2n\in 4\N+2$. Figure \ref{Fig3} shows the plots of the polynomials $p_{2n-1}(A)$ and $(2+A)p_{2n}(A)$ in one of such cases: \begin{figure}[h!]
	\centering
	\includegraphics[scale=1]{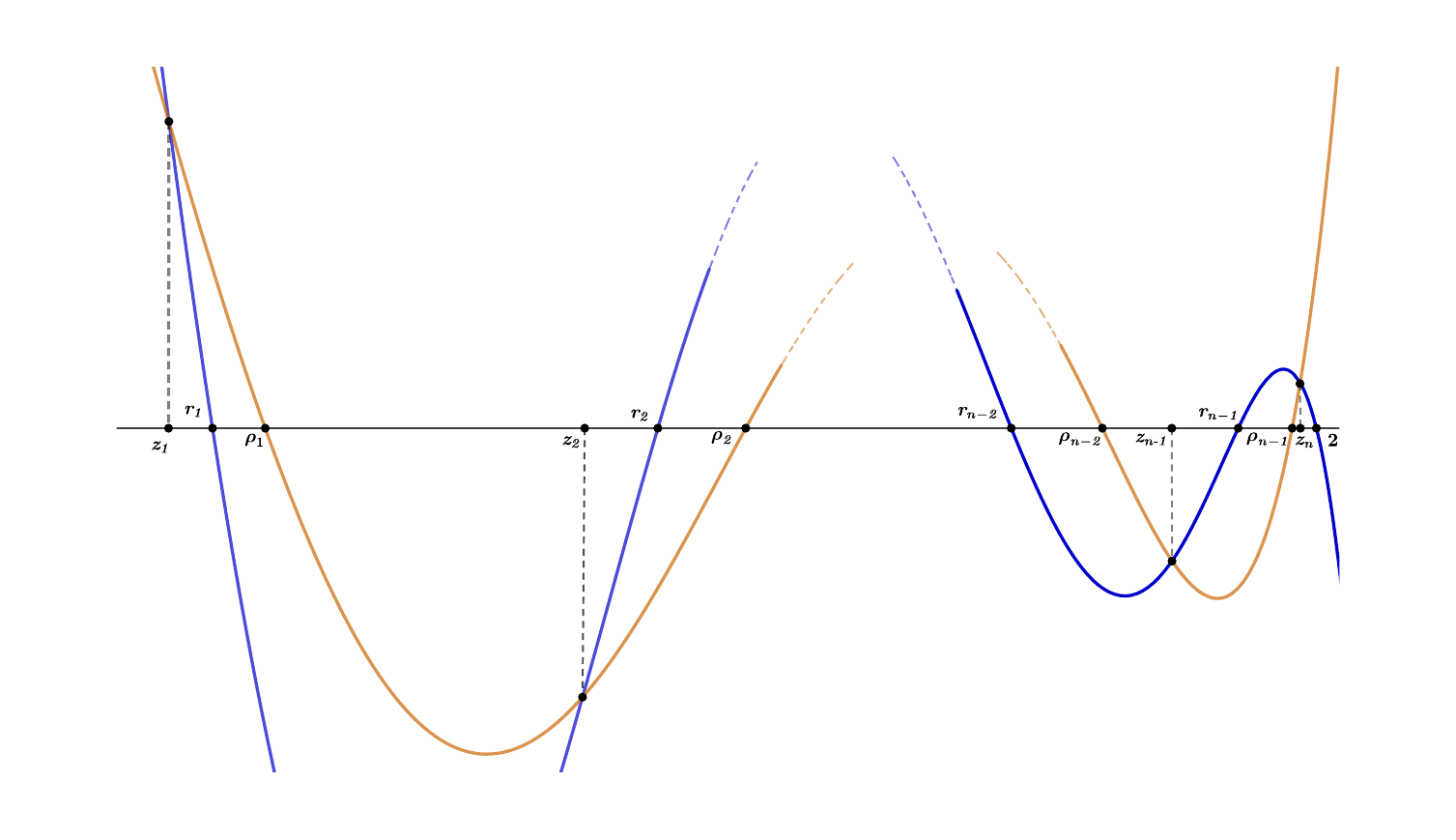}
	\caption{Sketch of the construction of  $p_{2n+1}(A)$.}
	\label{Fig3}
\end{figure}
$p_{2n-1}(A)$ has been plotted in brown and $(2+A)p_{2n}(A)$ in blue.
According to Lemmas \ref{le4.4} and \ref{le4.5}, we already know that
\begin{equation}
\label{5.2}
  2n-1 = p_{2n-1}(0)< 4n = 2p_{2n}(0), \quad \mathrm{deg\,}(p_{2n-1})=2n-2,\quad
  \mathrm{deg\,}((2+A)p_{2n})=2n,
\end{equation}
and, since $2n\in 4\N+2$, the leading coefficient of $p_{2n-1}(A)$ equals $1$, while the leading coefficient of $p_{2n}(A)$ equals $-1$. Thus, $p_{2n-1}(A)>0$ and $(2+A)p_{2n}(A)<0$ for $A>2$. By the induction assumption, the polynomials  $\tfrac{p_{2n}(A)}{2-A}$ and $p_{2n-1}(A)$ have (exactly) $n-1$ positive roots. Hence, each of the polynomials $p_{2n-1}(A)$ and $(2+A)p_{2n}(A)$ possesses (exactly) $n-1$ simple roots in the interval $(0,2)$ and, in addition, $p_{2n}(2)=0$. In Figure \ref{Fig3}, we have named by $\rho_i$, $1 \leq i \leq n-1$, the $n-1$ positive roots of $p_{2n-1}(A)$,
\[
  0 < \rho_1<\rho_2< \cdots < \rho_{n-2}<\rho_{n-1}<2,
\]
while those of $p_{2n}(A)$ less than $2$  have been named by  $r_i$, $1\leq i \leq n-1$. So,
\[
  0 < r_1< r_2<\cdots r_{n-1}<r_{n-1}< r_{n}:=2.
\]
As, again by the induction hypothesis, the positive roots of $(2+A)p_{2n}(A)$ are separated by the positive roots of $p_{2n-1}(A)$, necessarily
\begin{equation}
\label{5.3}
  0<r_1<\rho_1<r_2<\rho_2<\cdots < r_{n-2}<\rho_{n-2}<r_{n-1}<\rho_{n-1}<r_n=2.
\end{equation}
Consequently, by \eqref{5.1}, the polynomial $p_{2n+1}(A)$ must have, at least, $n$ different roots in the interval $(0,2)$. These roots have been named by $z_i$, $1\leq i \leq n$, in Figure \ref{Fig3} and they satisfy
\begin{equation}
\label{5.4}
  0 < z_1 < r_1 <  \rho_1 < z_2 < r_2 < \rho_2 < \cdots < z_{n-1}<
  r_{n-1}<\rho_{n-1}<z_n< 2.
\end{equation}
On the other hand, by  Corollary \ref{co4.9}, $p_{2n+1}(A)$ is an even polynomial. Thus, since, due to Lemma
\ref{le4.5}, it has degree $2n$ and, by the previous construction, $\pm z_i$, $1\leq i \leq n$, provides us
with a set of $2n$ different roots of $p_{2n+1}(A)$, necessarily
\[
  p_{2n+1}(A)= -\prod_{j=1}^n(A^2-z_j^2), \qquad A >0.
\]
Therefore, all the roots of $p_{2n+1}(A)$ are real and algebraically simple. As a direct consequence of
\eqref{5.4} it is apparent that the positive roots of $p_{2n+1}(A)$ are separated by the positive roots of $p_{2n}(A)$, except for $2$.
\par
Subsequently, we should prove the result in the special case when $2n\in 4\N$. In this situation,
owing to Lemmas \ref{le4.4} and \ref{le4.5}, the plots of the polynomials $p_{2n-1}(A)$ and $(2+A)p_{2n}(A)$ look like illustrated by Figure \ref{Fig4}. Apart from the fact that now $p_{2n-1}(A)>0$ and
$(2+A)p_{2n}(A)<0$ for all $A>2$, because the leading coefficients change sign, the previous analysis can be easily adapted to cover the present situation in order to infer that $p_{2n+1}(A)$ satisfies all the requirements also in this case. By repetitive the technical details of the proof are omitted here in.
\begin{figure}[h!]
	\centering
	\includegraphics[scale=1]{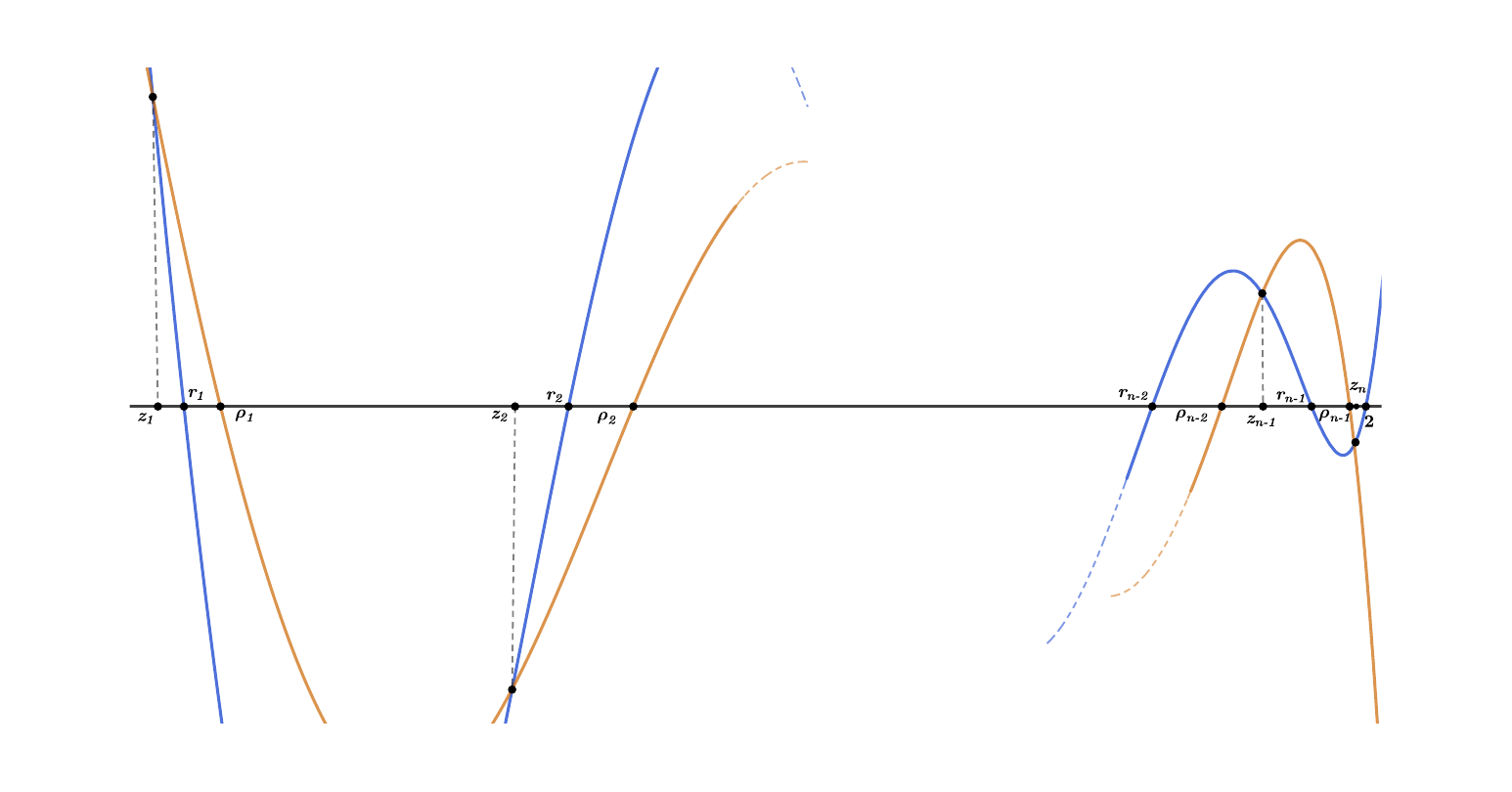}
	\caption{Sketch of the construction of $ p_{2n+1}(A) $. }
	\label{Fig4}
\end{figure}
\vspace{0.4cm}
\par
\noindent\textbf{Step 2: Passing from $p_{2n+1}(A)$ to $p_{2n+2}(A)$, $n\geq 2$.} According to Figure \ref{Fig2}, it becomes apparent that the two positive roots of $p_5(A)$ are separated by the unique root of $p_4(A)$ less than $2$. Moreover, all their roots are real and simple. Note that the polynomials
\[
 \frac{p_4(A)}{2-A}= -A^2+2, \qquad p_5(A)=A^4-5A^2+5,
\]
have one and two positive roots respectively. Arguing by induction, assume that $p_{2n}(A)$ and $p_{2n+1}(A)$ satisfy all the requirements in the statement of the theorem  for some $n\geq 2$, i.e., all the positive roots of these polynomials are real and algebraically simple, the positive roots of $p_{2n+1}(A)$ are separated by the positive roots less than $2$ of $p_{2n}(A)$, and the polynomials  $\tfrac{p_{2n}(A)}{2-A}$ and $p_{2n+1}(A)$ have, respectively, $n-1$  and $ n $ positive roots. We claim that the roots of the polynomial $p_{2n+2}(A)$ are real and simple, that they are separated by the roots of $p_{2n+1}(A)$, and that $p_{2n+2}(A)$ possesses $n+1$ positive roots. Indeed, by Theorem \ref{th4.8} ,
\begin{equation}
\label{5.5}
  p_{2n+2}(A)=(2-A)p_{2n+1}(A) - p_{2n}(A).
\end{equation}
As in the previous step, we first deal with the case when $2n\in 4\N+2$. By Lemmas \ref{le4.4} and \ref{le4.5}, we already know that
\begin{equation}
\label{5.6}
  2n = p_{2n}(0)< 4n+2 = 2p_{2n+1}(0), \;\;  \mathrm{deg\,}(p_{2n})=2n-1,\;\;
  \mathrm{deg\,}\left((2\!-\!A)p_{2n+1}\right)=2n+1,
\end{equation}
and, since $2n\in 4\N+2$, the leading coefficient of $p_{2n}(A)$ equals $-1$, and the leading coefficient of $p_{2n+1}(A)$ equals also $-1$. Thus,
\[
  p_{2n}(A)<0,\qquad (2-A)p_{2n+1}(A)>0\qquad \hbox{for all}\;\; A>2.
\]
By the induction assumption, the polynomials  $\tfrac{p_{2n}(A)}{2-A}$ and $p_{2n+1}(A)$ have (exactly) $n-1$ and $n$ positive roots, respectively. Thus, each of the  polynomials $p_{2n}(A)$ and $(2-A)p_{2n+1}(A)$ possesses (exactly) $n$ simple roots in $(0,2)$ and, obviously, $(2-A)p_{2n+1}(A)$  also vanishes at $A=2$. Figure \ref{Fig5} shows the plots of $p_{2n}(A)$, in blue, and $(2-A)p_{2n+1}(A)$, in brown.
\begin{figure}[h!]
	\centering
	\includegraphics[scale=1]{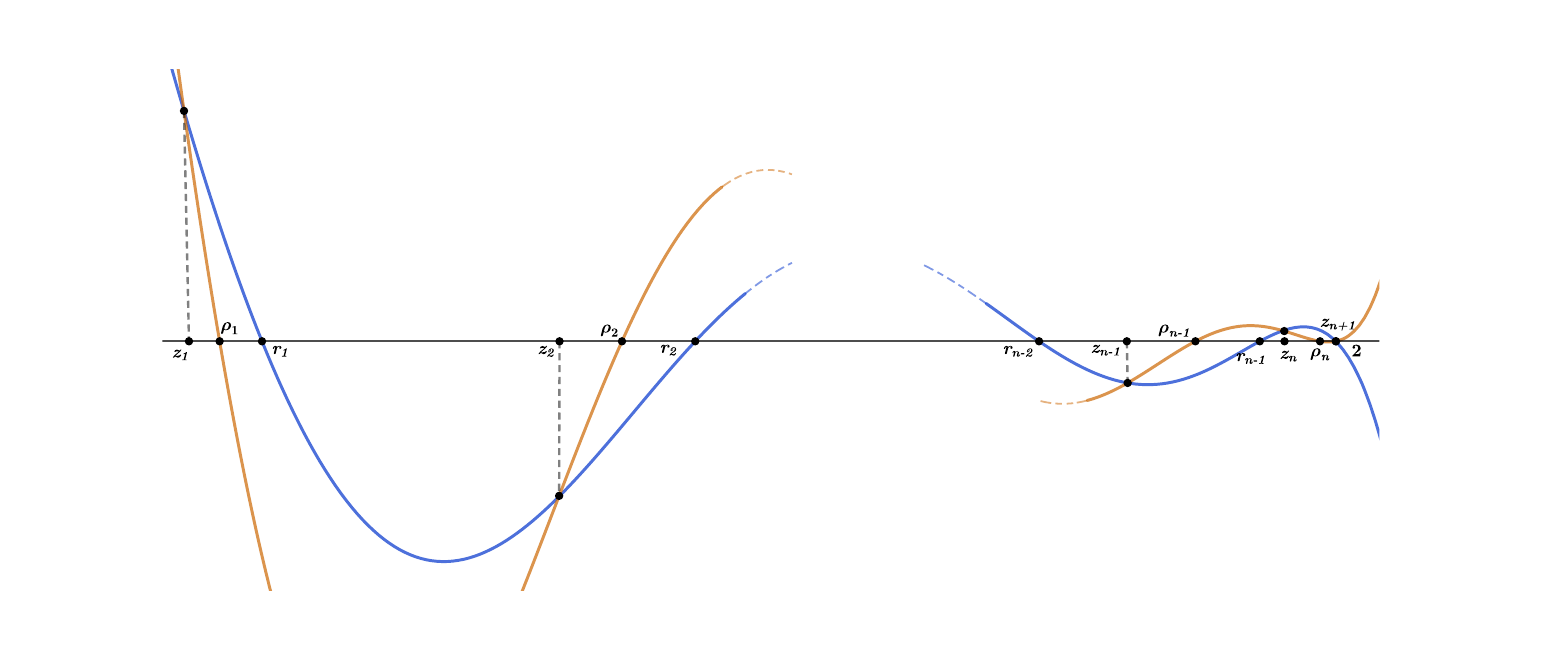}
	\caption{Sketch of the construction of $ p_{2n+2}(A) $ in case $2n\in 4\N+2$. }
	\label{Fig5}
\end{figure}
In Figure \ref{Fig5}, we have named by $\rho_i$, $1 \leq i \leq n$, the $n$ positive roots less than $2$ of $(2-A)p_{2n+1}(A)$,
\[
  0 < \rho_1<\rho_2< \cdots < \rho_{n-1}<\rho_{n}<2:=\rho_{n+1},
\]
whereas $r_i$, $1\leq i \leq n$, stand for the positive roots of $p_{2n}(A)$. Since $p_{2n}(2)=0$,
$r_n=2$. Since the positive roots of $p_{2n+1}(A)$ are separated by the positive roots less than $2$ of
$p_{2n}(A)$, the following holds
\[
  0 < \rho_1<r_1<\rho_2<r_2< \cdots < \rho_{n-1}<r_{n-1}<\rho_{n}<2:=\rho_{n+1}=r_{n},
\]
as illustrated by Figure \ref{Fig5}. Thanks to \eqref{5.5}, it becomes apparent that the polynomial
$p_{2n+2}(A)$ admits, at least, an interior root in each of the intervals $(\rho_i,\rho_{i+1})$, $i=0,...,n$, denoted by $z_i$ in Figure \ref{Fig5}, plus $z_{n+1}=2$. Here we are setting $\rho_0:=0$. Consequently,
$p_{2n+2}(A)$ has, at least, $n+1$ positive roots.
\par
On the other hand, thanks to Corollary \ref{co4.9},  $\tfrac{p_{2n+2}(A)}{2-A} $ is an even function and hence, $p_{2n+2}(A)$ has, at least, $2n+1$ different roots. Since, by Lemma \ref{le4.5},
\[
  \mathrm{deg\,}(p_{2n+2})=2n+1,
\]
all these roots are real and algebraically simple. By construction, it is apparent that the positive roots of $p_{2n+2}(A)$ are separated by the positive roots of $p_{2n+1}(A)$ (see Figure \ref{Fig5} if necessary).
\par
If, instead of $2n\in 4\N+2$, we impose $ 2n\in 4\N $, then the previous arguments can be easily adapted to complete the proof of the  theorem from Figure \ref{Fig6}, where the graphs of $(2+A)p_{2n+1}(A)$ and $p_{2n}(A)$ have been superimposed in order to show their crossing points, which, owing to
 Theorem \ref{th4.8} , are the roots of $p_{2n+2}(A)$.  By repetitive, the technical details of this case are not included here.
\end{proof}	
\begin{figure}[h!]
	\centering
	\includegraphics[scale=1]{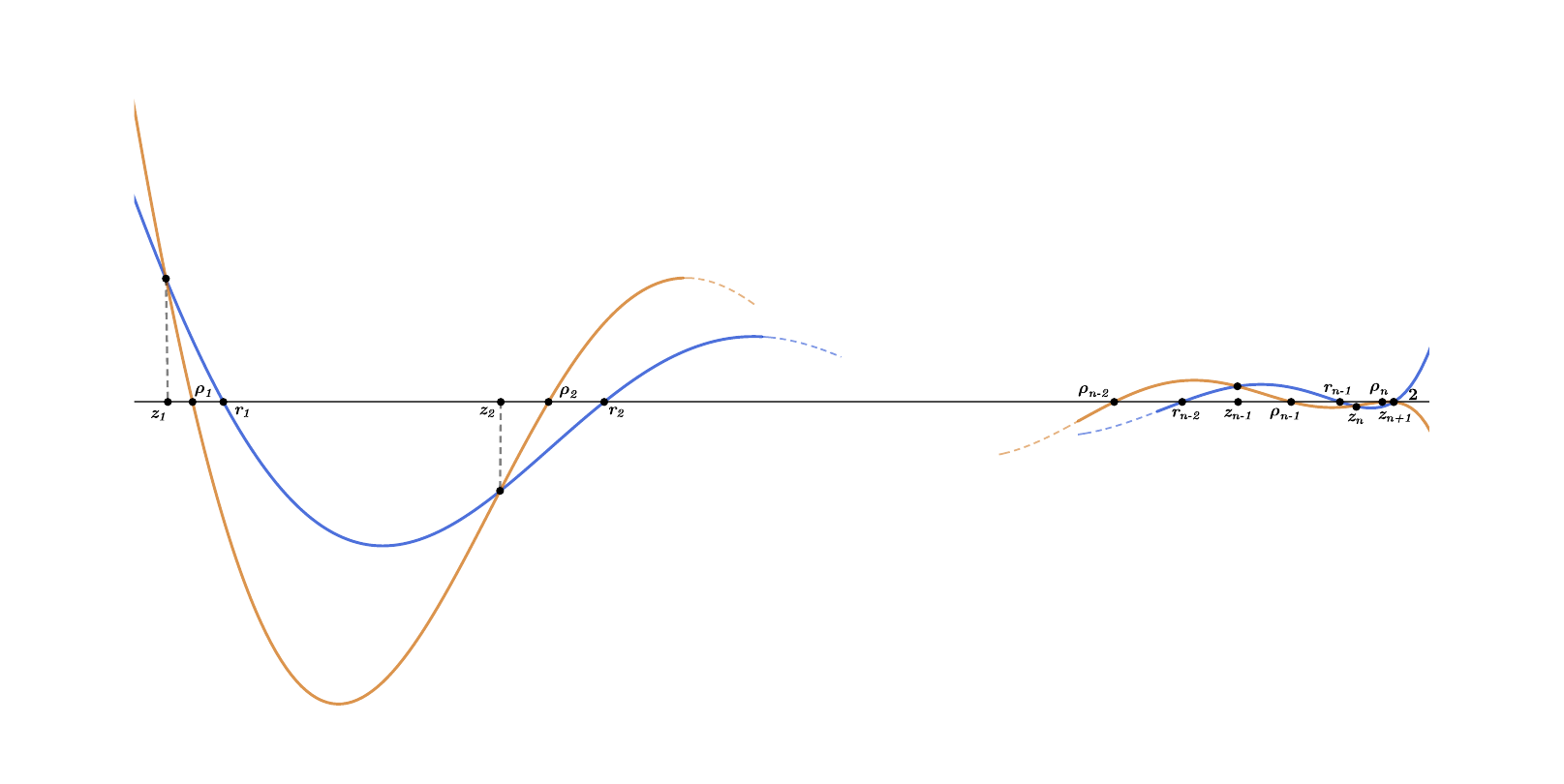}
	\caption{Sketch of the construction of $ p_{2n+2}(A) $ in case $2n\in 4\N$. }
	\label{Fig6}
\end{figure}

A careful reading of the proof of Theorem \ref{5.1} reveals that, actually, not only the roots of
$p_{n}(A)$ are separated by those of $p_{n-1}(A)$, but that they are also separated by those of $p_{n-2}(A)$, taking always into account the exceptional role played by the root $2$.
\par

\section{Global bifurcation diagram}
	
This section analyzes the global structure of the set of zeros of the maps $\v_n$, $n\geq 1$, introduced in \eqref{3.8}. These zeros are the positive fixed points of the Poincar\'e maps $\mc{P}_n$, $n\geq 1$,  constructed in Section 3. They provide us with the $nT$-periodic
coexistence states of \eqref{1.1} under the additional assumption \eqref{3.5}.  It should be remembered that,  according to \eqref{4.2},  for every integer $n\geq 1$
\begin{equation}
\label{6.1}
  p_n(A)=  \frac{d\v_n (A,1)}{dx}=\mf{L}(n;A),\qquad A>0,
\end{equation}
provides us with the linearization at the \emph{trivial curve}, $(A,1)$, of $\v_n(A,x)$. In our analysis, $A$ is always regarded as a bifurcation parameter to  $nT$-periodic coexistence states from the $T$-periodic ones (i.e., from $x=1$). As a consequence of the simplicity of all the roots of $p_n(A)$, $n\geq 1$, guaranteed by Theorem \ref{th5.2}, the following result holds.

\begin{Thm}
\label{th6.1} For every $n\geq 1$ and $r \in p_n^{-1}(0)$ the following algebraic transversality condition holds
\begin{equation}
\label{6.2}
		\mathfrak{L}_1 (N\left[\mathfrak{L}(n;r) \right] )\oplus R\left[ \mathfrak{L}(n;r) \right]=\mathbb{R},
\end{equation}
where
\[
  \mathfrak{L}_1:= \frac{d \mathfrak{L}(n;r)}{d A} , \qquad n\geq 1,\;\; r\in p_n^{-1}(0).
\]
Therefore, by Theorem 1.7 of M. G. Crandall and P. H. Rabinowitz \cite{CR71}, there exists an
analytic curve of $nT$-periodic coexistence states of \eqref{1.1} bifurcating from $(A,1)$ at the root $A=r$. Actually, there exists $\e>0$ and a real analytic map $A: (-\e,\e)\to \R$ such that $A(0)=r$ and
\[
  \v_n(A(s),1+s)=0 \quad \hbox{for all}\;\; s\in (-\e,\e).
\]
Moreover, any non-trivial zero of $\v_n$, $(A,x)$ with $x\neq 1$,  in a neighborhood of $(r,1)$ must be of the form $(A(s),1+s)$ for some $s \in (-\e,\e)$. In other words, there exists $\varrho >0$ such that
\[
  \left. \begin{array}{rr} \v_n(A,x)=0 \\ |A-r|+|x-1|<\varrho \\ x\neq 1
  \end{array}\right\} \Longrightarrow   (A,x)=(A(s),1+s)  \quad \hbox{for some}\;\; s \in (-\e,\e).
\]
Furthermore, setting
\begin{equation}
\label{6.3}
  A(s)= r + A_1 s + A_2 s^2 + \mathcal{O} (s^3) \quad \hbox{as}\;\; s\to 0,
\end{equation}
one has that
\begin{enumerate}
\item[{\rm (a)}] $A_1=0$ and $A_2>0$ if $n=2$ and $r=2$, in complete agreement with Figure \ref{Fig1};
\item[{\rm (b)}] $A_1<0$ if $n=3$ and $r=\sqrt{3}$;
\item[{\rm (c)}] $A_1=0$ and $A_2 <0$ (resp. $A_2>0$) if $n=4$ and $r = r_{4,1}=\sqrt{2}$
(resp. $r = r_{4,2}=2$).
\end{enumerate}
\end{Thm}
\begin{proof}
According to \eqref{6.1}, $\mf{L}(n;r)=p_n(r)=0$. Thus, $N[\mf{L}(n;r)]=\R$ and \eqref{6.2} can be equivalently expressed as $\mf{L}_1(\R) =\R$, which holds true because, thanks to Theorem \ref{5.1}, we already know that  $r$ is an algebraically simple root of $p_n(A)$, i.e.,
\[
  \mf{L}_1 = p_n'(r) \neq 0.
\]
So, \eqref{6.2} indeed holds and \cite[Th. 1.7]{CR71} applies to $\v_n(A,x)=0$ at $(A,x)=(r,1)$.
Since we can take $\psi=1$ as a generator of $N[\mf{L}(n;r)]=\R$ and $Y=[0]$ as a supplement of $N[\mf{L}(n;r)]=\R$ in $\R$, owing to \cite[Th. 1.7]{CR71}, there exist $\e>0$ and  a real analytic map
\[
  (A,y):(-\e,\e)\to \R\times Y
\]
such that $(A(0),y(0))=(r,0)$ and
\begin{equation}
\label{6.4}
  \v_n(A(s),1+s(\psi+y(s)))=0 \quad \hbox{for all}\;\; s\in (-\e,\e),
\end{equation}
it becomes apparent, by construction, that
\[
  \v_n(A(s),1+s)=0\quad \hbox{for all} \;\; s \in (-\e,\e),
\]
because $y\equiv 0$ and $\psi=1$. This ends the proof of the first two claims of the theorem:
the existence of the analytic curve of nontrivial solutions and the uniqueness.
\par
As far as concerns to the problem of ascertaining the nature of these \emph{local bifurcations} at
$(r,1)$, we can proceed as follows. In order to prove Part (a), note that, thanks to \eqref{6.4}, setting $x(s):=1+s$ and expanding in Taylor series, we have that
\[
 0 = \v_2(A(s),x(s))=\v_2(r,1)+\frac{d\v_2}{ds}(r,1)s+\frac{1}{2}\frac{d^2\v_2}{ds^2}(r,1)s^2+\cdots
\]
for all $s\in(-\e,\e)$, where $r=2$. Moreover, by  construction, we already know that
\[
  \v_2(r,1)=0,\qquad \frac{\p\v_2}{\p x}(r,1)=p_2(r)=p_2(2)=0
\]
(see \eqref{6.1}, if necessary). Thus, since by \eqref{3.13}
\begin{equation}
\label{6.5}
  \v_2(A,x)=x\left( E_1(A,x)+1\right) -2,
\end{equation}
it follows from \eqref{6.3} and $\frac{\partial E_1}{\partial A}(r,1)=0$ that
\[
   \frac{d\v_2}{ds}(r,1)=\frac{\p\v_2}{\p A}(r,1)A'(0)=\frac{\partial E_1}{\partial A}(r,1)A_1=0,
\]
where $':=\frac{d}{ds}$. Hence, these terms do not provide us with any neat information concerning $A_1$.
So, we must consider higher order terms to find out $A_1$. As
\[
  \frac{\p\v_2}{\p A}(r,1)=0=\frac{\p\v_2}{\p x}(r,1),
\]
applying the chain rule it readily follows that
\begin{equation}
\label{6.6}
  0= \frac{d^2\v_2}{ds^2}(r,1)=\frac{\p^2\v_2}{\p x^2}(r,1)+2\frac{\p^2\v_2}{\p A\p x}(r,1)A_1+\frac{\p^2\v_2}{\p A^2}(r,1)A_1^2.
\end{equation}
On the other hand, differentiating with respect to $x$ the identity \eqref{6.5} yields
\begin{equation}
\label{6.7}
  \frac{\p \v_2}{\p x}(A,x) = E_1(A,x)+1+ x \frac{\p E_1}{\p x}(A,x).
\end{equation}
So,
\[
  \frac{\p^2 \v_2}{\p x^2}(A,x) = 2 \frac{\p E_1}{\p x} (A,x)+ x \frac{\p^2 E_1}{\p x^2}(A,x).
\]
Consequently, particularizing at $(A,x)=(r,1)$, it follows from \eqref{3.10} that
\begin{equation}
\label{6.8}
  \frac{\p^2 \v_2}{\p x^2}(r,1) =r^2-2r=4-4=0.
\end{equation}
Similarly, differentiating \eqref{6.7} with respect to $A$ shows that
\[
  \frac{\p^2 \v_2}{\p x\p A}(A,x)=\frac{\p E_1}{\p A}(A,x)+x\frac{\p^2 E_1}{\p x\p A}(A,x)
\]
and hence, owing to \eqref{3.10},
\begin{equation}
\label{6.9}
   \frac{\p^2 \v_2}{\p x\p A}(r,1)=\frac{\p E_1}{\p A}(r,1)+\frac{\p^2 E_1}{\p x\p A}(r,1) =
   \frac{\p^2 E_1}{\p x\p A}(r,1) = -1.
\end{equation}
Lastly,
\[
  \frac{\p^2\v_2}{\p A^2}(A,x)=x\frac{\p^2 E_1}{\p A^2}(A,x)= (1-x)^2E_1(A,x)
\]
and hence,
\begin{equation}
\label{6.10}
  \frac{\p^2\v_2}{\p A^2}(A,1)=0.
\end{equation}
Therefore, substituting \eqref{6.8}, \eqref{6.9} and \eqref{6.10} into \eqref{6.6} it becomes apparent that $A_1=0$. Thanks to this fact, the third derivative admits the next (simple) expression:
\[
   0=\frac{d^3\v_2}{ds^3}(r,1)=6\frac{\p^2\v_2}{\p x\p A}(r,1)A_2+\frac{\p^3\v_2}{\p x^3}(r,1)=-6A_2+3r^2-r^3,
\]
which implies  that
\[
  A_2=\frac{4}{3}>0
\]
and ends the proof of Part (a).
\par
To prove Part (b), note that, much like in Part (a), one has that
\[
   0= \v_3(A(s),x(s))=\v_3(r,1)+\frac{d\v_3}{ds}(r,1)s+\frac{1}{2} \frac{d^2\v_3}{ds^2}(r,1)s^2+\cdots
\]
for all $s \in (-\e,\e)$. Similarly,
\[
   \frac{\p \v_3}{\p A}(r,1)=0=\frac{\p\v_3}{\p x}(r,1).
\]
So,
\[
  \frac{d\v_3}{ds}(r,1)=0.
\]
Moreover, differentiating twice with respect to $s$ yields
\begin{align*}
  0 & = \frac{d^2\v_3}{ds^2}(r,1)\\[5pt] & =\frac{\p^2\v_3}{\p x^2}(r,1)+2\frac{\p^2\v_3}{\p x\p A}(r,1)A_1+\frac{\p^2\v_3}{\p A^2}(r,1)A_1^2\\[5pt] & =r^4-r^3-2r^2-4rA_1.
\end{align*}
Consequently, since $r=\sqrt{3}$, it follows from this identity that
\[
  A_1=\frac{\sqrt{3}-3}{4}<0,
\]
which ends the proof of Part (b).
\par
Finally, much like before, we have that
\[
   0= \v_4(A(s),x(s))=\v_4(r,1)+\frac{d\v_4}{ds}(r,1)s+\frac{1}{2}\frac{d^2\v_4}{ds^2}(r,1)s^2+\cdots
\]
for all $s \in (-\e,\e)$, and, in addition,
\[
 \frac{\p \v_4}{\p A}(r,1)=0=\frac{\p \v_4}{\p x}(r,1).
\]
Thus, $\frac{d\v_4}{ds}(r,1)=0$.  Moreover, differentiating twice yields
\begin{align*}
     0 & =\frac{d^2\v_4}{ds^2}(r,1)\\[5pt] &=\frac{\p^2\v_4}{\p x^2}(r,1)+2\frac{\p^2\v_4}{\p x\p A} (r,1)A_1+\frac{\p^2\v_4}{\p A^2}(r,1)A_1^2 \\[5pt] &=r^6-3r^5-r^4+8r^3-2r^2-4r+2(r^3-2r^2-2r)A_1.
\end{align*}
Therefore, since $r=\sqrt{2}$ it follows from this identity that $A_1=0$. Furthermore,
\begin{align*}
0 & = \frac{d^3\v_2}{ds^3}(r,1)\\[5pt] &=6\frac{\p^2\v_2}{\p x\p A} (r,1)A_2+\frac{\p^3\v_2}{\p x^3}(r,1)\\[5pt] &=6(3r^2-4r-2)A_2-r^8+r^7+9r^6-11r^5-10r^4+20r^3-2r^2.
\end{align*}
Consequently, we find from  $r=\sqrt{2}$ that
\[
A_2=-\frac{2(5+4\sqrt{2})}{3}<0,
\]
which ends the proof.
\end{proof}

Figure \ref{Fig7} shows the local bifurcation diagrams of the $2T$, $3T$ and $4T$-periodic coexistence states of \eqref{1.1} under condition \eqref{3.5}. We are plotting $x$, in ordinates, versus $A$, in abscisas.
By the analysis already done at the beginning of Section 2, and, in particular, by Theorem \ref{th2.1}, which was sketched in Figure \ref{Fig1}, we already know that, under condition \eqref{3.5}, the problem \eqref{1.1} admits a $2T$-periodic coexistence state if, and only if, $A>2$. Moreover, the local bifurcation of these solutions must be supercritical. Thus $A_2\geq 0$. As a byproduct of Theorem \ref{th6.1}, it turns out that $A_2>0$. So, it is  a genuine supercritical pitchfork bifurcation of quadratic type.
However, since $A_1<0$, the bifurcation to $3T$-periodic coexistence states from $(A,x)=(\sqrt{3},1)$ is transcritical, whereas the  $4T$-periodic solutions emanate from $(A,x)=(\sqrt{2},1)$ through a
subcritical quadratic pitchfork bifurcation, because $A_1=0$ and $A_2<0$ in this case.
\par
The fact that the local nature of the first three bifurcation phenomena possess a completely different character  shows that, in general,   ascertaining the precise type of these local bifurcations for large $n$ might not be possible, much like happened with the problem of determining the fine structure of the set of bifurcation points from the trivial solution $(A,1)$. The higher is the order of the bifurcating subharmonics, measured by $n$,   the higher is the complexity of the associated function $\v_n$ and hence, the more involved is finding out  the values of $A_1$ and $A_2$ in \eqref{6.3} by the intrinsic nature of the functions $E_n$ defined in \eqref{3.10}.

\begin{figure}[h!]
	\centering
	\includegraphics[scale=0.72]{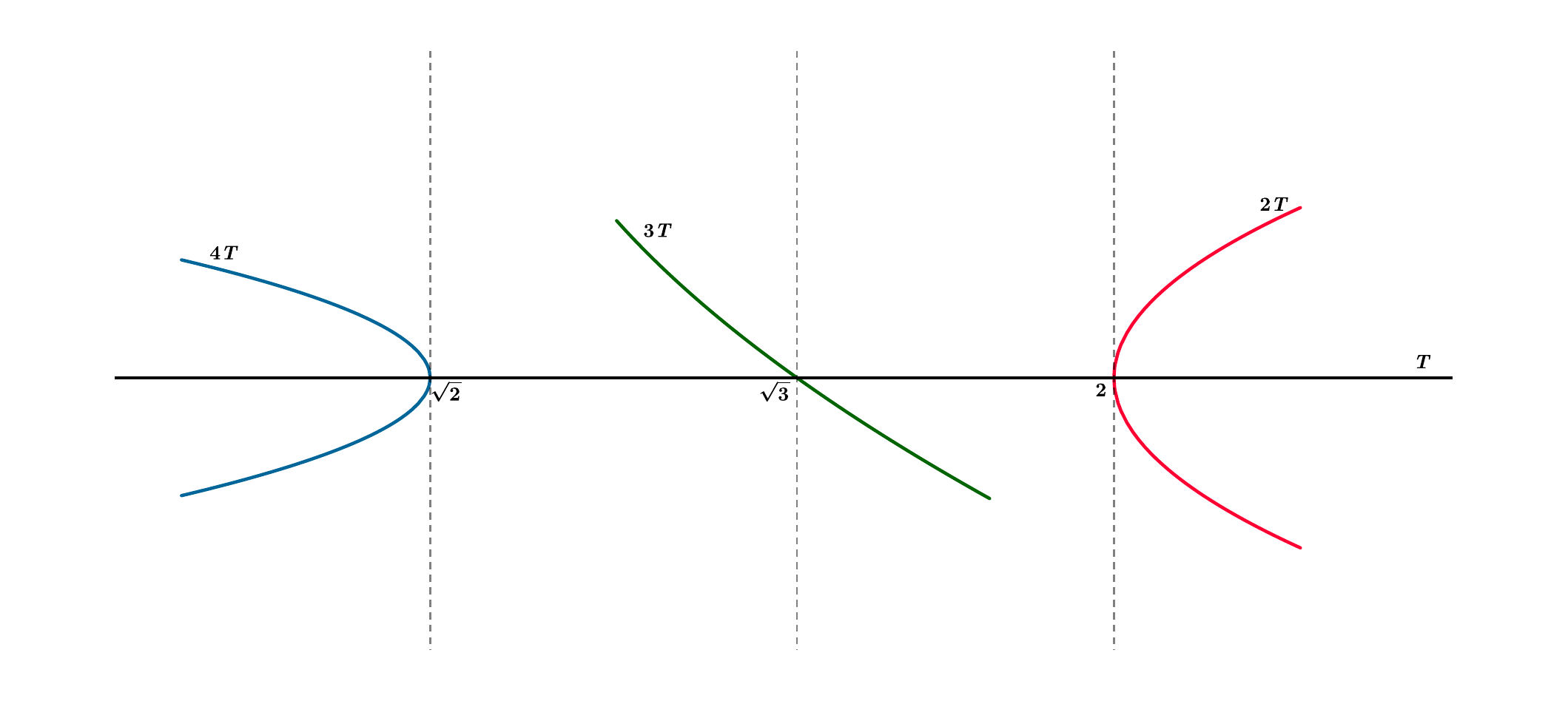}
	\caption{Local bifurcation diagrams from $(A,x)=(A,1)$ of the $nT$-periodic coexistence states for $n\in\{2,3,4\}$.}
	\label{Fig7}
\end{figure}

\begin{Rem}
\label{re6.2}
Thanks to Theorem \ref{th6.1}, it becomes apparent that the set of bifurcation points from $ (A,1) $ to $ nT $-periodic coexistence states of \eqref{1.1} is the set of roots of $ p_n(A) $. Since the number of roots of a polynomial is finite, the set of bifurcation points is numerable, as it is a numerable union of finite sets.
\par
Since every  $nT$-periodic coexistence state of \eqref{1.1} provides us with a $knT$-periodic coexistence state for all $k\geq 1$, owing Theorem \ref{th6.1}, the roots of $p_n(A)$ must be roots of $p_{kn}(A)$ for all  $ n,k\geq 1 $, i.e., $p_n|p_{kn}$ for all $n, k \geq 1$.
\end{Rem}
	
\begin{Rem}
\label{re6.3}
Thanks to Theorem \ref{th5.2} and Remark \ref{re6.2}, the set of bifurcation points to a $nT$-periodic solution is a subset of the interval  $(0,2]$. Complementing \cite[Th. 5.2]{LG00}, where the non-degeneration of the positive $T$-periodic coexistence states of \eqref{1.1}  with respect to the $T$-periodic solutions was established,  Theorem \ref{th6.1} shows that the $T$-periodic solutions are degenerated with respect to the $nT$-periodic solutions of \eqref{1.1} for all $n\geq 2$ at every positive root, $r$, of $p_n(A)$.
Nevertheless, the $T$-periodic solutions are non-degenerated with respect to the $nT$-periodic solutions, $n\geq 2$,  if $A>2$, because in this range there is not any bifurcation point from $(A,1)$.
\end{Rem}

Subsequently, we will discuss the global character of all the local bifurcations documented by Theorem
\ref{th6.1} in the context of global bifurcation theory. In this discussion, by a (connected) component it is understood any closed and connected subset that is maximal for the inclusion. For any given integer $n\geq 1$, the set of non-trivial $nT$-periodic solutions of \eqref{1.1}, $\mc{S}_n$,  consists of all $nT$-periodic coexistence states different from $(A,1)$ plus the set of points $(r,1)$ with $p_n(r)=0$. In other words, setting $\R_+:=(0,+\infty)$,
\[
  \mc{S}_n= \{(A,x) \in \R_+\times (\R_+\setminus\{1\}) \;:\;  \v_n(A,x)=0\}\cup\{(r,1)\;:\; p_n(r)=0\}.
\]
Note that, owing to Theorem \ref{th6.1}, $\{(r,1)\;:\; p_n(r)=0\}$ is the set of bifurcation points to $nT$-periodic solutions from the trivial curve $(A,1)$. Thanks to \eqref{6.2},  the algebraic multiplicity of J. Esquinas and J. L\'opez-G\'omez \cite{ELG87} equals one,
\[
  \chi[\mf{L}(n;A);r]=1 \in 2\N +1,
\]
for all $r \in p_n^{-1}(0)$, $r>0$. Thus, by \cite[Th. 5.6.2]{LG01},
the local degree at $(A,1)$ of the one-dimensional $\v_n(A,\cdot)$ changes as $A$ crosses $r$ (see also J. L\'opez-G\'omez and
C. Mora-Corral \cite[Pr. 12.3.1]{LGM07} if necessary). Therefore, according to \cite[Th. 6.2.1]{LG01},
for every integer $n\geq 2$ and each root $r>0$ of $p_n(A)$, there is a component  of $\mc{S}_n$, $\mf{C}_{n,r}$, such that
\[
  (r,1) \in \mf{C}_{n,r}\subset \R_+\times\R.
\]
Moreover, by the local uniqueness about  $(r,1)$ guaranteed by Theorem \ref{th6.1}, in a neighborhood of $(r,1)$ the component $\mf{C}_{n,r}$ consists of an analytic curve, $(A(s),1+s)$, $|s|<\e$. Note  that any real continuous map must be compact. So, the Leray--Schauder degree (see, e.g., N. G. Lloyd \cite{Llo}, or
 \cite[Ch. 12]{LGM07}, if necessary)  can be applied to get these global results. Alternatively, one might use the degree of P. Benevieri and M. Furi \cite{BF98}, as in Theorem 5.4 and Corollary 5.5 of \cite{LGM05} (see \cite{LG16} for a recent survey on global bifurcation theory).
\par
By Remark \ref{re3.4}, for every $r \in p_n^{-1}(0)\cap \R_+$, the component $\mf{C}_{n,r}$ must be separated away from $x=0$ and hence, all their solutions must be positive, because $\v_n(0)=-n$. Thus, they indeed provide us with coexistence states of \eqref{1.1}. Similarly, for every $n\geq 2$, since $\v_n(n)>0$, $\mf{C}_{n,r}$ is bounded above by $n$, in the sense that $x<n$ if $(A,x)\in \mf{C}_{n,r}$ with $A>0$. Therefore,
\begin{equation}
\label{6.11}
  \mc{P}_x(\mf{C}_{n,r})\subset (0,n),
\end{equation}
where $\mc{P}_x$ stands for the $x$-projection operator,  $\mc{P}_x(A,x):=x$. Moreover, due to \eqref{4.8}, $x=1$ is the unique zero of $\v_n(A,x)$ at $A=0$. Note that, due to Remark \ref{re6.2},  $(A,1)=(0,1)\notin \mf{C}_{n,r}$ because $p_n(0)=n>0$.
\par
Throughout the rest of this section, we will also consider the (unilateral) subcomponents
\begin{equation}
\label{6.12}
  \mf{C}_{n,r}^+:=\mf{C}_{n,r}\cap [x>1], \qquad \mf{C}_{n,r}^-:=\mf{C}_{n,r}\cap [x<1].
\end{equation}
Thanks to Theorem \ref{th6.1}, these subcomponents are non-empty. Moreover, arguing as in \cite[p. 182]{LG01}, it is easily seen that they equal the components $\mf{C}^+$ and $\mf{C}^-$ introduced on page \cite[p. 187]{LG01}. This feature heavily relies on the fact that $x$ is a one-dimensional variable.  Therefore, the unilateral theorem \cite[Th. 6.4.3]{LG01} can be applied to infer that each of the components $\mf{C}_{n,r}^+$ and $\mf{C}_{n,r}^-$ satisfies the \emph{global alternative} of P. H. Rabinowitz
\cite{Rab}, because the supplement of $N[\mf{L}(n;r)]=\R$ in $\R$ is $Y=[0]$ and, due to \eqref{6.11}, $\mf{C}_{n,r}$ cannot admit an element, $(A,x)$ with $x=0$. Therefore, $\mf{C}_{n,r}^+$ (resp.
$\mf{C}_{n,r}^-$) satisfies some of the following two conditions, which are far from being excluding:
\begin{enumerate}
\item[(a)] There exists $s \in p_n^{-1}(0)\setminus\{r\}$ (resp. $t \in p_n^{-1}(0)\setminus\{r\}$) such that $(s,1)\in \mf{C}_{n,r}^+$ (resp. $(t,1)\in \mf{C}_{n,r}^-$).
\item[(b)] The component $\mf{C}_{n,r}^+$ (resp. $\mf{C}_{n,r}^-$) is unbounded in $A$, because of \eqref{6.11}.
\end{enumerate}
Note that the counterexample of E. N. Dancer \cite{Dan} shows that Theorems 1.27 and 1.40 of P. H. Rabinowitz \cite{Rab} are not true as originally stated. To show that the second option occurs in both cases  we need the next result.

\begin{Lem}
\label{le6.4} Each of the unilateral subcomponents satisfies
\begin{equation}
\label{6.13}
  \mf{C}_{n,r}^\pm \cap \{(A,1)\;:\; A\geq 0\} = \{(r,1)\}.
\end{equation}
Thus, also
\begin{equation*}
  \mf{C}_{n,r} \cap \{(A,1)\;:\; A\geq 0\} = \{(r,1)\},
\end{equation*}
i.e., $(r,1)$ is the unique bifurcation point of $\mf{C}_{n,r}$ from $(A,1)$.
\end{Lem}
\begin{proof} Subsequently, we will denote by $\nu(n)$ the total number of positive roots of the polynomial $p_n(A)$. By Theorem \ref{5.2}, we already know that $\nu(n)=\frac{n}{2}$ if $n$ is even and $\nu(n)=\frac{n-1}{2}$ if $n$ is odd. We will prove the result only for $\mf{C}_{n,r}^+$, as the same argument also works out to prove the corresponding assertion for the component $\mf{C}_{n,r}^-$. The proof will proceed by contradiction. We already know that $\mf{C}_{n,r}^+$ can only meet the trivial solution $(A,1)$ at the roots of $p_n(A)$. Suppose that $r=r_{n,i}$ for some $i \in \{1,...,\nu(n)\}$, and that there exists $j>i$, $j \in \{1,...,\nu(n)\}$, such that
\begin{equation}
\label{6.14}
  \{(r_{n,i},1),(r_{n,j},1)\} \subset \mf{C}_{n,r}^+ \cap \{(A,1)\;:\; A\geq 0\}.
\end{equation}
Then, by the definition of component, it becomes apparent that
\begin{equation}
\label{6.15}
  \mf{C}_{n,r_{n,i}}^+=\mf{C}_{n,r_{n,j}}^+
\end{equation}
as sketched  by Figure \ref{Fig8}. Thanks to Theorem \ref{th5.2},
there exists
\[
  r_{n-1,k}\in (r_{n,i},r_{n,j})\cap p_{n-1}^{-1}(0).
\]
By the incommensurability of $nT$ with $(n-1)T$, $\mf{C}_{n-1,r_{n-1,k}}^+$ cannot reach the component
\eqref{6.15}. Thus, must be bounded. Consequently, as $\mf{C}_{n-1,r_{n-1,k}}^+$  also satisfies the global alternative of P. H. Rabinowitz, there exists $r_{n-1,\ell}\in p_{n-1}^{-1}(0)$, with $k\neq \ell$, such that
\[
  \{(r_{n-1,k},1),(r_{n-1,\ell},1)\} \subset \mf{C}_{n-1,r_{n-1,k}}^+ \cap \{(A,1)\;:\; A\geq 0\},
\]
as sketched in Figure \ref{Fig8}. Since the set of roots
\[
  \bigcup_{2\leq \kappa \leq n} p_{\kappa}^{-1}(0)
\]
is finite, it becomes apparent that, after finite many steps, there exists a component,  $\mf{C}_{n-h,r_{n-h,m}}^+$, for some $2 \leq h \leq n-3$ and $1\leq m\leq \nu(n-h)$, that should meet
the last component  linking two different roots sketched in Figure \ref{Fig8},
\[
  \mf{C}_{n-h+1,r_{n-h+1,v}}^+ = \mf{C}_{n-h+1,r_{n-h+1,w}}^+,
\]
because there is no any additional root of $p_{n-h}(A)$ in between $r_{n-h+1,v}$ and $r_{n-h+1,w}$.
But this is impossible, by the incommensurability of $(n-h)T$ with $(n-h+1)T$. This contradiction
ends the proof.
\end{proof}

\begin{figure}[h!]
	\centering
	\includegraphics[scale=1.2]{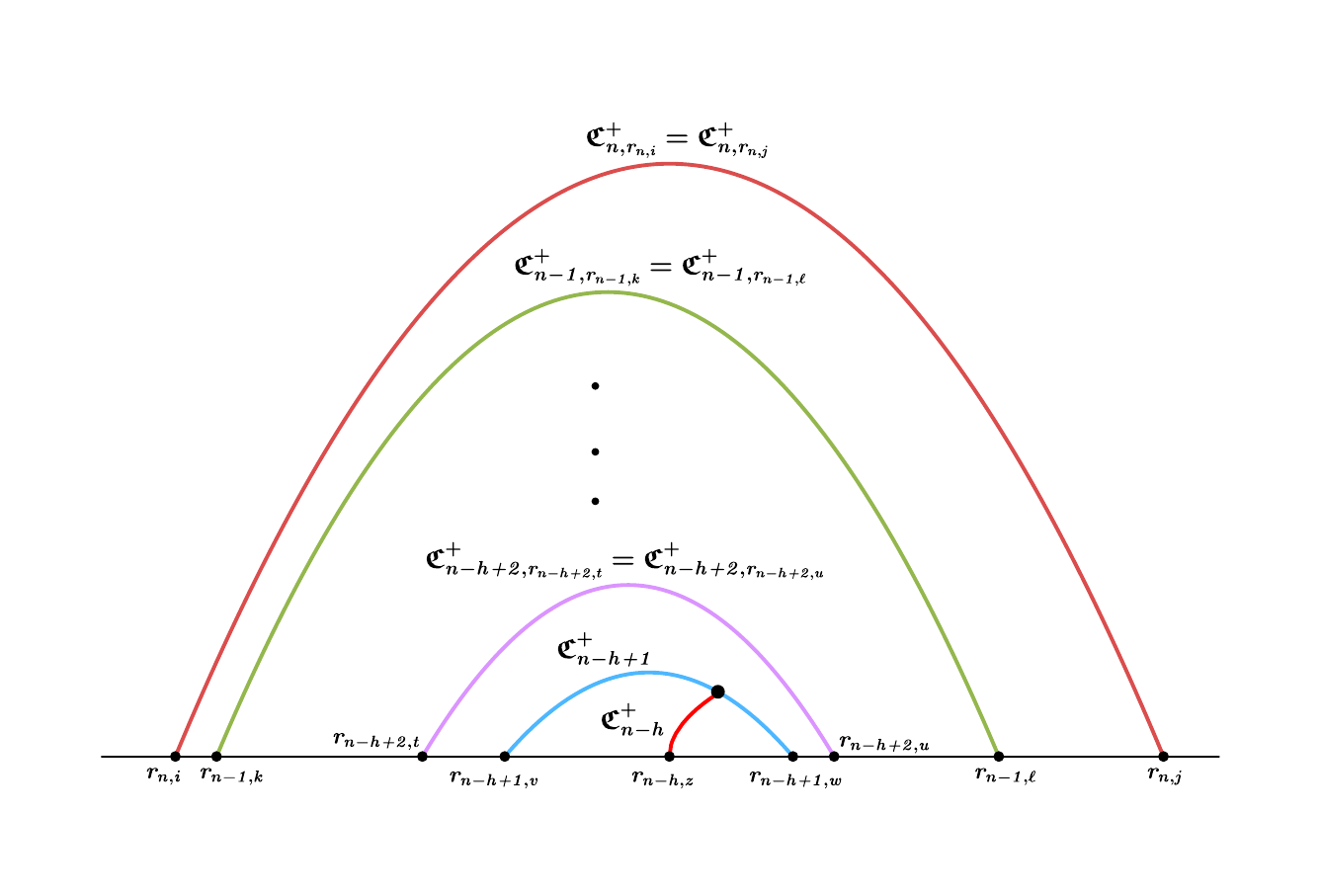}
	\caption{Sketch of the proof of Lemma \ref{6.3}.}
	\label{Fig8}
\end{figure}

As an immediate consequence of the previous analysis, the next result holds. As for the $x$-projection operator, $\mc{P}_x$, we will denote by $\mc{P}_A$ the $A$-projection operator,
\[
  \mc{P}_A(A,x):=A.
\]

\begin{Thm}
\label{th6.5}
For every integer $n\geq 2$ and each root $r>0$ of $p_n(A)$, the component $\mf{C}_{n,r}^+$ satisfies
\begin{enumerate}
\item[{\rm (a)}] $\mc{P}_x (\mf{C}_{n,r}^+)\subset [1,n)$;
\item[{\rm (b)}] $\mc{P}_A (\mf{C}_{n,r}^+)=[A^+_{n,r},+\infty)$ for some $A_{n,r}^+ \in (0,r]$. In particular, $\mf{C}_{n,r}^+$ is unbounded.
\item[{\rm (c)}] $\mf{C}_{n,r}^+ \cap \{(A,1)\;:\; A\geq 0\} = \{(r,1)\}$;
\item[{\rm (d)}] For every $n, m \geq 2$, $\mf{C}_{n,r}^+\cap \mf{C}_{m,s}^+=\emptyset$ if $r \neq s$.
\end{enumerate}
Moreover, by Theorem \ref{th6.1}, in a neighborhood of $(r,1)$ the component $\mf{C}_{n,r}^+$ consists of an analytic curve, $(A(s),1+s)$, $0\leq s <\e$.  Similarly,
the component $\mf{C}_{n,r}^-$ satisfies {\rm (c)}, {\rm (d)} and
\begin{enumerate}
\item[{\rm (A)}] $\mc{P}_x (\mf{C}_{n,r}^-)\subset (0,1]$;
\item[{\rm (B)}] $\mc{P}_A (\mf{C}_{n,r}^-)=[A^-_{n,r},+\infty)$ for some $A_{n,r}^- \in (0,r]$. In particular, $\mf{C}_{n,r}^-$ is unbounded.
\end{enumerate}
Analogously,  in a neighborhood of $(r,1)$ the component $\mf{C}_{n,r}^-$ consists of an analytic curve, $(A(s),1+s)$, $-\e<s\leq 0$.
\end{Thm}
\begin{proof}
At this stage, the only delicate point is Part (d). Suppose that
$\mf{C}_{n,r}^+\cap \mf{C}_{m,s}^+\neq \emptyset$ for some $r \neq s$. Then, by the definition of component, necessarily
\[
  \mf{C}_{n,r}^+ = \mf{C}_{m,s}^+.
\]
Thus, $(r,1), (s,1) \in \mf{C}_{n,r}^+$, which contradicts Lemma \ref{le6.4}. The proof is complete.
\end{proof}

Except for the local bifurcations from the trivial line $(A,1)$, the global diagramas of
the components $\mf{C}_{n,r}^\pm$ plotted in Figure \ref{Fig9} respect the general properties
established by Theorem \ref{th6.5}. Although the components have been plotted with no secondary bifurcations along them, there are some numerical evidences that $\mf{C}_{2,2}^-$ possesses a secondary bifurcation to $4T$-periodic solutions. Nevertheless, thanks to Theorem \ref{th6.5}, even in the case that they might occur higher order bifurcations along these components, they must be disjoint.
\par
According to Theorems \ref{th6.1} and \ref{th6.5}, it becomes apparent  that some $3T$ and $4T$-periodic solutions must be degenerated. Namely, those on the turning points of $\mf{C}_{3,\sqrt{3}}^+$, $\mf{C}_{4,\sqrt{2}}^+$ and $\mf{C}_{4,\sqrt{2}}^-$ in Figure \ref{Fig7}. Similarly, the bifurcation points accumulating from the left to $\sqrt{2}$ and  $\sqrt{3}$ must provide us with additional degenerate solutions: those on the turning points of
their corresponding components.

\end{document}